\documentclass[12pt,reqno,a4paper]{amsart}
\usepackage{mathtools}
\usepackage{enumerate}
\usepackage{comment}   
\usepackage{graphicx} 
\usepackage{amsmath,amssymb,amsfonts,amsthm}
\usepackage{xcolor}
\usepackage[
inner=2.3cm,
outer=2.3cm,
top=2.7cm,
marginparwidth=2.5cm,
marginparsep=0.1cm
]
{geometry}
\allowdisplaybreaks

\usepackage{pgfplots}
\usepackage{tikz}
\usetikzlibrary{math,calc}

\newcommand{\ve}{\varepsilon}
\newcommand\cL{{\mathcal L}}
\newcommand\bP{{\mathbb P}}
\newcommand\cC{{\mathcal C}}
\newcommand\cP{{\mathcal P}}
\newcommand\bE{{\mathbb E}}
\newcommand\cF{{\mathcal F}}
\newcommand\bR{{\mathbb R}}

\newcommand{\LL}{{\mathrm{LL}}}
\newcommand{\Lip}{{\mathrm{Lip}}}

\newtheoremstyle{special}%
{}%
{}%
{}%
{}%
{\scshape}%
{.}%
{.5em}%
{}
\newtheorem{theorem}{Theorem}

\newtheorem{lemma}[theorem]{Lemma}
\newtheorem{corollary}[theorem]{Corollary}
\newtheorem{proposition}[theorem]{Proposition}
\newtheorem{remark}[theorem]{Remark}
\newtheorem{definition}[theorem]{Definition}
\newtheorem{assumption}[theorem]{Assumption}

\def\Id{\text{\rm Id}}

\DeclareMathOperator{\esssup}{esssup}
\DeclareMathOperator{\essinf}{essinf}
\begin{document}
	
	\title{Decay of correlations and limit theorems for random intermittent maps }
\author{Davor Dragi\v cevi\' c}
\thanks{Faculty of Mathematics, University of Rijeka, Radmile Matej\v ci\' c 2, 51000 Rijeka, Croatia.\\ Email: ddragicevic@math.uniri.hr}

\author{Yeor Hafouta} \thanks{Department of Mathematics, Ben-Gurion University  of the Negev, Israel. \\ Email: 
yeorh@bgu.ac.il}

\author{Juho Lepp{\"a}nen} \thanks{Department of Mathematics, Tokai University, Kanagawa 259-1292, Japan\\ Email: leppanen.juho.heikki.g@tokai.ac.jp}

\subjclass{Primary: 37D25, 37H05}
	
	\maketitle

	
	

    \begin{abstract}
     In this paper, we revisit the problem of polynomial memory loss and the central limit theorem for time-dependent LSV maps. More precisely, we show that for random LSV maps corresponding to a random parameter $\beta(\cdot)$ we obtain quenched memory loss,  decay of correlations, central limit theorems with rates, moment bounds and almost sure invariance principles (ASIP) when the essential infimum of $\beta(\cdot)$ is sufficiently small\footnote{DD: I have replaced $<\frac 1 5$ with sufficiently small here}
and the driving process (i.e. random environment) is mixing sufficiently fast. In \cite[Corollary 3.8]{Su2} the ASIP was obtained for ergodic driving systems when the essential supremum of $\beta$ is less than $1/2$. As will be elaborated in Section \ref{Intro}, restrictions on the essential infimum are more natural in our context.  Our results have an abstract form which we believe could be useful in other circumstances, as will be elaborated in a future work. 
    \end{abstract}
    \section{Introduction}\label{Intro} 
An important discovery made in the last century is that autonomous expanding (or hyperbolic) dynamical systems could exhibit stochastic behavior. One of the most celebrated results in this direction is the fact that appropriately normalized Birkhoff sums could satisfy the central limit theorem (CLT). Since then many other probabilistic limit theorems have been obtained for autonomous systems, and we refer to \cite{GH,GO, HH} for spectral approaches that yield a variety of limit theorems for autonomous (partially) hyperbolic dynamical systems.   
However, many systems appearing in nature are non-autonomous  due to an interaction with the outside world. 
Such systems can be better described by compositions of
different maps, i.e. time-dependent transformations, rather than by repeated application of exactly the same transformation, 
so that the $j$-th iterate of the system is given by 
$
T_{j-1}\circ\ldots\circ  T_{1}\circ T_0.
$
Yet, many powerful tools developed for studying autonomous systems are unavailable in non-autonomous setting (e.g. spectral theory of a single quasi-compact operator, see \cite{HH}), so very often new ideas
are needed to handle the non-stationary case.


One notable example of non-autonomous systems are random dynamical systems.
Random transformations emerge in a natural way as a model
for the description of a physical system whose evolution mechanism depends on
time in a stationary way. This leads to the study of the actions of compositions
of different maps $T_j$ chosen at random from a typical sequence of transformations. To fix the notation, this means that there is an underlying probability preserving system
$(\Omega,\mathcal F,\mathbb P,\sigma)$ so that the $n$-th step evolution of the system is given by 
$$
T_\omega^n:=T_{\sigma^{n-1}\omega}\circ\ldots\circ T_{\sigma\omega}\circ T_\omega.
$$
This setup was already discussed by Ulam and von Neumann~\cite{UN} and Kakutani~\cite{Ka} in connection with random ergodic theorems. 
The ergodic theory of random dynamical systems has attracted a lot of attention in the past decades, see, for instance \cite{Arnold98, Cong97, Crauel2002, Kifer86, KiferLiu, LiuQian95}. We refer to the introduction of \cite[Chapter 5]{KiferLiu} for a historical discussion and applications to, for instance, statistical physics, economy, and meteorology.  


Probabilistic limit theorems (aka statistical properties) of random dynamical systems have attracted a lot of attention in the literature over the past two decades. For example, the decay of the correlations was obtained in \cite{ABR, ABRV, Baldi, Buzzi}. The central limit theorem for iterations of random mappings chosen
independently with identical distribution (iid) hyperbolic transformations was considered in \cite{ANV, ALS09}. In this case the orbits in state space
form a Markov chain (\cite{Kifer86}) and the limit theorems are obtained relying on stationary methods which involve the spectral gap of an appropriate annealed transfer operator. Another approach for iid transformation and for small perturbations of deterministic systems is based on construction of random Young towers \cite{ABR,Baldi,Su1}, which exist only in particular situations, and their implementation seem to heavily rely on independence of the maps or on a perturbative approach. We stress that in the perturbative case the results require exponential tails for the towers, and in that case many limit theorems follow from general results for random Young towers driven by ergodic systems, see \cite{YH YT}. 

 Beyond iterations of random iid maps, limit theorems were mostly obtained for quite general classes of random  expanding transformations  and for some classes of random uniformly hyperbolic maps $T_\omega$, see, for instance \cite{DFGTV, DavorCMP, DavorTAMS, DavH AHP, HK, Nonlin, YH YT, YH Advances, YH cones 2, Kifer 1996, Kifer, Ste20, STE,STSU} and references therein. In fact, in the uniformly hyperbolic case, there are also a few results for compositions of non-random sequences of maps \cite{Bhak,DolgHaf PTRF 2}. In particular, the results in \cite{Bhak} cover certain types of random uniformly hyperbolic maps, and the results in \cite{DolgHaf PTRF 2} cover most of the results for random uniformly expanding maps.

In this paper, we study random expanding intermittent maps $T_\omega$ on the unit interval, that is, we assume that there is one indifferent fixed point (say $0$) such that $T_\omega'(0)=1$ and around $0$ we have $T_\omega'(x)=1+O(x^{\beta(\omega)})$ for some $\beta(\omega)\in(0,1)$. In fact, to simplify the arguments, we will work with the classical LSV model \cite{LSV} for such maps, where $T_\omega=T_{\beta(\omega)}$ for
\[
	T_\beta(x)=\begin{cases}
		x(1+2^\beta x^\beta) & 0\le x<\frac 1 2 \\
		2x-1 & \frac 1 2 \le x\le 1.
	\end{cases}
	\]
    The difficulty here compared with the expanding case when $\inf_{x\in[0,1]}|T_\omega'(x)|>1$ (perhaps even non-uniformly in $\omega$ as in \cite{YH Advances, YH cones 2}) is that the point $0$ attracts a non-negligible part of the unit interval. This has the effect of ``destroying" the exponential mixing and replacing it with a polynomial one. Already for deterministic maps when $T_\omega=T_\beta$ for some fixed $\beta$ this makes the classical Nagaev-Guivarch spectral method \cite{HH} inapplicable. 
    
    Another difficulty that arises is that, already in the deterministic case, the polynomial rates are not achieved in some operator norm, and what we get is what in this context is referred to as ``loss of memory", which means that we can control the mixing rates by means of $L^s$-norms ($s<\infty$) of the iterations of the transfer operator applied to Lipschitz continuous functions. To fix the notation, if we denote by $L_\omega$ the transfer operator of $T_\omega$ with respect to the underlying random equivariant measures $\mu_\omega$, then what we mean is that for $\bP$-a.e. $\omega \in \Omega$, for every Lipschitz $\varphi_1,\varphi_2:[0,1]\to\mathbb R$ 
\begin{align}\label{p control}
			\begin{split}
				& \bigl\Vert  \bigl[  L_{  \sigma^i \omega }^{j-i} \bigl(  L_{ \sigma^r \omega }^{i-r}   ( \varphi_1 )  \varphi_2\bigr)  \bigr]_{ \sigma^j \omega }  \bigr\Vert_{  L^s(  \mu_{ \sigma^j \omega } ) }
				\le C    (1 +  \Vert  \varphi_1 \Vert_{ \Lip } )  (1 +  \Vert  \varphi_2 \Vert_{ \Lip } )
				K( \sigma^i \omega )  (j - i)^{-a},
			\end{split}
    \end{align}
    where $K\in L^p(\Omega,\mathcal F,\mathbb P)$ for some $p$ large enough, $\|g\|_{Lip}$ is the usual  Lipschitz  norm, $L_\omega^j=L_{\sigma^{j-1}\omega}\circ\cdots\circ L_{\omega}$ and $[g]_{\omega}=g-\mu_{\omega}(g)$.
    In the random LSV case, $a$ can be expressed using $p,s$ and the essential infimum of $\beta(\cdot)$, but our abstract results will be obtained under \eqref{p control} with some $a$.
    The fact that one can only control the $L^s$-norms is a serious obstacle from an analytic point of view, even compared to the case $s=\infty$ (which is excluded here). In fact, a large part of the proof for random LSV maps is to show that \eqref{p control} holds. Once this is established the proofs of all the limit theorems follow from \eqref{p control} and martingale methods. We believe that \eqref{p control} is also true for random Young towers (see the arguments in the proof of \cite[Lemma 5.9]{Su1}), but usually such towers extension exist for iid transformations or for close maps, which is against the main point in this paper, where weakly dependent maps are considered and the random parameter $\beta$ is allowed to take values in the entire set of parameters $(0,a]$ for some $a<1$. 

    The almost sure invariance principle (ASIP) is a powerful statistical tool as it allows us to couple the underlying sequence $S_n$ with a Brownian motion in such a way that $|S_n-B_{\|S_n\|_{L^2}}|=O(\|S_n\|_{L^2}^{1-\varepsilon})$ a.s. where $\varepsilon>0$ is some small constant. Clearly, it implies the central limit theorem (CLT), but it also implies the functional CLT, the functional law of iterated logarithm and many other limit theorems (see \cite[Appendix C]{BCS}).
    In \cite{Su} Su developed an approach to prove ASIP in the former circumstances (i.e., when $s<\infty$), relying on the Skorokhod embedding theorem. Note that this is really needed in these circumstances, compared with existing tools. For example, in \cite{CM} the authors developed a method (which is also) based on reverse martingale couboundary representations to get the ASIP. The results in \cite{CM} were applied for a variety of expanding systems, both stationary and time-dependent (random or sequential; see \cite{DFGTV,Nicol,YH Advances}). However, a closer look reveals that the conditions in \cite{CM} do not seem to hold when $s<\infty$. This is where the method of Su comes in handy. We stress that more recently, Su~\cite{Su2} has established ASIP for vector-valued observables and improved the parameter restrictions of his earlier work~\cite{Su}.

    In this paper, we consider random LSV maps and prove polynomial loss of memory (again with every $s<\infty$). Then we adapt the method in \cite{Su} and get the ASIP for different classes of random LSV maps than in \cite{Su,Su2}. More precisely, in the random case Su (see~\cite{Su2}) requires $\beta:=\|\beta(\cdot)\|_{L^\infty}<1/2$ (in~\cite{Su} it was required that $\beta <\frac 1 8$). However, philosophically, in the case of a random dynamical system, a more natural restriction on the random parameter $\beta(\omega)$ should involve upper bounds on the essential infimum of $\beta(\omega)$ and not on its essential supremum. In fact, as $\beta$ decreases, the map $T_\beta$ resembles $T_0$, which is the classical doubling map, so an inducing argument yields the result that a positive proportion of the maps $T_{\sigma^n\omega}$ would be close to $T_0$. 
    We will follow the latter heuristics and prove the ASIP without restrictions on $\beta$ apart from $\beta<1$, and instead we will prove the ASIP when the essential infimum $\gamma$ of $\beta(\omega)$ is less than $1/3$.
For example, the distribution of $\beta(\omega)$ can be equivalent to the uniform distribution on $(0,a)$ for some $a<1$, or just $\mathbb P(\beta(\omega)=1/3-\varepsilon)>0$ for some $\varepsilon>0$, and many other examples can be given. The ``price" here is that we assume that the random environment $(\Omega,\mathcal F,\mathbb P,\sigma)$ is mixing sufficiently fast in an appropriate sense (and not only ergodic as in~\cite{Su, Su2}). 
As a by-product of our methods, we also prove moment bounds, CLT rates, and similar results for the skew product. We also prove the CLT under the optimal assumption that $\gamma<1/2$. 


As noted above, our first step is to prove polynomial loss of memory (and decay of correlations) when starting with Lipschitz observables, which can be viewed as initial densities. Compared with the memory loss estimates in \cite[Corollary 3.3]{KL2}, we obtain more information on the behavior of the random multiplicative constants that appear
in the estimates. Specifically, we establish \eqref{p control} with $K\in L^p(\Omega,\mathcal F,\mathbb P)$ provided that, for $b > 0$, the tail probabilities
$$
\bP\biggl(  
| S_n^{ \pm  }(\omega) - \bE[ S_n^{ \pm  }(\cdot)  ]  |\ge nb
\biggr) 
$$
decay at a sufficiently fast polynomial rate (depending on $p$) as $n \to \infty$. Here, $S_n^{+}(\omega)$ (resp. $S_n^{-}(\omega)$) denotes the number of parameters in the random sequence $( \beta(\omega), \ldots, \beta( \sigma^{n-1} \omega ) )$
(resp. in $( \beta( \sigma^{-n + 1}  \omega ), \ldots, \beta( \omega ) )$) that do not exceed $\gamma=\essinf(\beta(\omega))$.  This observation
enables us to use large deviations estimates from the literature on mixing random sequences to reduce the condition $K\in L^p(\Omega,\mathcal F,\mathbb P)$ to a condition involving $\gamma$ and the speed of mixing of the driving process. We also establish moment bounds for Birkhoff sums of the random dynamical system
in the regime $\gamma < 1/2$, which extend previous results in
\cite{NPT, KL, Su2}.
To our knowledge, the ASIP, CLT rates, and moment bounds for random LSV maps obtained in this paper are all new. 
We emphasize that previous studies on CLTs and invariance principles for random LSV maps, including 
\cite{Su, Su1, Su2, NPT, BQT, LSVW, LS}, either assume an i.i.d. driving process or impose conditions 
on the essential supremum of $\beta(\cdot)$ beyond $\Vert \beta(\cdot) \Vert_{L^\infty} < 1$.

More precisely, the previous results concerned with limit theorems for random compositions of LSV assume either that $\|\beta (\cdot)\|_{L^\infty}<\frac 1 2$ as in~\cite{DL,NPT,NTV,Su,Su2} (we note that some works like~\cite{NTV} impose even more restrictive conditions) or assume that the driving system is i.i.d and allow that $\beta(\omega)$ is sampled from a larger interval~\cite{BB, BQT, Su1}. In sharp contrast to the former series of papers, in the present paper we only require $\|\beta(\cdot)\|_{L^\infty}<1$, while compared to the latter series of papers our results are not restricted to the case where the driving system is i.i.d.  (see Remark~\ref{remdriving}). In particular, even our quenched CLT for random compositions of LSV maps (see Corollary~\ref{cltlsv}) is a completely new result and is obtained under the optimal condition $\gamma<1/2$. The same (regarding originality) applies to quenched ASIP (Corollary~\ref{asipLSVnew}), where the previous result of Su~\cite{Su1} gives ASIP under the assumption that the driving system is i.i.d and that $\beta(\omega)$ is supported on $[\alpha_0, \alpha_1]$ where $0<\alpha_0<\frac 1 6$, $\alpha_1<1$.
	
\section{Quenched limit theorems}
Throughout this paper, $(\Omega,\mathcal F,\mathbb P,\sigma)$ is an ergodic probability-preserving system.
Let $M$ be a metric space, and let $M_ \omega, \omega\in \Omega$  be random measurable closed subsets; namely, 
we assume that 
\[
\mathcal M:=\{(\omega, x): \ \omega \in \Omega, \ x\in M_\omega \} \in \mathcal{F} \otimes \mathcal{B}(M),
\]
where $\mathcal{B}(M)$ denotes the Borel $\sigma$-algebra of $M$.
Let $T_\omega:M_\omega\to M_{\sigma\omega},\, \omega\in\Omega$, be a measurable family of maps (that is, the map $(\omega,x)\to T_\omega(x)$ is measurable). The Borel $\sigma$-algebra on $M_\omega$ will be denoted by $\mathcal B_\omega$. Let us assume that there are Borel probability measures $\mu_\omega$ on $M_\omega$ such that $(T_\omega)_*\mu_\omega=\mu_{\sigma\omega}$ for $\mathbb P$-a.e. $\omega \in \Omega$. Let $L_\omega$ be the transfer operator of $T_\omega$ with respect to these measures, namely for all bounded measurable functions $g:M_\omega\to\mathbb R$ and $f:M_{\sigma\omega}\to\mathbb R$ we have 
$$
\int_{M_\omega} g\cdot (f\circ T_\omega)\,d\mu_\omega=\int_{M_{\sigma \omega }} (L_\omega g) \cdot f\,d\mu_{\sigma_\omega}.
$$
Denote 
\begin{equation}\label{Ln}
	L_\omega^n=L_{\sigma^{n-1}\omega}\circ\ldots\circ L_{\sigma\omega}\circ L_{\omega}, \quad (\omega, n)\in \Omega \times \mathbb N.
\end{equation}
\begin{assumption}\label{P assum}
	There exists a finite set $\mathcal I \subset (0, \infty)$, a random variable $K\in L^p(\Omega,\mathcal F,\mathbb P)$ with $p>0$ and for each $s\in \mathcal I$ a decreasing sequence $(a_{s, n})_n$ of positive numbers converging to $0$ such that for $\mathbb P$-a.e. $\omega\in \Omega$,  every $s\in \mathcal I$, $r\leq i\leq j$ and Lipschitz functions $g_1 \colon M_{\sigma^r\omega}\to\mathbb R$ and $g_2\colon M_{\sigma^i\omega}\to\mathbb R$  we have 
	\begin{align}\label{DEC}
		\begin{split}
			& \bigl\Vert  \bigl[  L_{  \sigma^i \omega }^{j-i} \bigl(g_2  L_{ \sigma^r \omega }^{i-r}   (g_1 \bigr)  \bigr) \bigr]_{ \sigma^j \omega }  \bigr\Vert_{  L^s(  \mu_{ \sigma^j \omega } ) }
			\\
			&\le  K( \sigma^i \omega )a_{s, j-i} (1 +  \Vert  g_1 \Vert_{ \Lip } )  (1 +  \Vert g_2 \Vert_{ \Lip } ),
		\end{split}
	\end{align}
	where $\|g\|_{\Lip}$ is the usual Lipschitz norm and $[g]_\omega=g-\mu_{\omega}(g)$. In particular, 
	\begin{equation}\label{dec2}
		\|L_\omega^ng\|_{L^s(\mu_{\sigma^n \omega})}\le K(\omega)a_{s, n}(1+\|g\|_{\Lip}),
	\end{equation}
	for $\mathbb P$-a.e. $\omega \in \Omega$, $n\in \mathbb N$ and $g\colon M_\omega \to \mathbb R$ Lipschitz with $\int_{M_\omega}g\, d\mu_\omega=0$.
\end{assumption}
\begin{remark}
In the sequel, when we apply the previous assumption in the case where $\mathcal I=\{s\}$ we will write $a_n$ instead of $a_{s, n}$.
\end{remark}
In the sequel, we will prove quenched limit theorems for random variables of the form $S_n^\omega \varphi=\sum_{j=0}^{n-1}\varphi_{\sigma^j\omega}\circ T_\omega^j$ when viewed as random variables on the probability space $(M_\omega, \mathcal B_\omega ,\mu_\omega)$ where 
$$
T_\omega^j=T_{\sigma^{n-1}\omega}\circ\ldots\circ T_{\sigma\omega}\circ T_{\omega}.
$$
This will be done in the case when Assumption~\ref{P assum} holds with sufficiently large $p$ and $a_{s,n}$ that decays sufficiently fast to $0$.

We also consider the skew-product transformation $\tau \colon \mathcal M\to \mathcal M$ defined by 
\begin{equation}\label{spt}
	\tau(\omega, x)=(\sigma \omega, T_\omega(x)), \quad (\omega, x)\in \mathcal M.
\end{equation}
The map $\tau$ encodes both the dynamics on the base space governed by $\sigma$, and the dynamics on fibers $M_\omega$ governed by  maps $T_\omega$. 
Let $\mu$ be a probability measure on $\mathcal M$ given by
\begin{equation}\label{measmu}
	\mu(\mathcal C)=\int_\Omega \mu_\omega(\mathcal C_\omega)\, d\mathbb P(\omega)
\end{equation}
for $\mathcal C\subset \mathcal M$ measurable (with respect to the trace
$\sigma$-algebra $[\mathcal{F}\otimes\mathcal{B}(M)]|_{\mathcal M}$), where 
\[
\mathcal C_\omega:=\{x\in M_\omega: \ (\omega, x)\in \mathcal C\}.
\]
Here we assume that
the map $\omega\mapsto\mu_\omega(\mathcal C_\omega)$ is measurable for each $\mathcal C\subset \mathcal M$ measurable.

\begin{proposition}\label{properg}
	Suppose that Assumption~\ref{P assum} holds with $\mathcal I=\{s\}$ with $s\ge 1$. Then $\mu$ is ergodic for $\tau$.
\end{proposition}
\begin{proof}
	Take a measurable subset $\mathcal C\subset \mathcal M$ such that $\tau^{-1}(\mathcal C)=\mathcal C$. We claim $\mu(\mathcal C)\in \{0, 1\}$. As in the proof of~\cite[Proposition 7]{DH-N}, we have 
	\begin{equation}\label{inv}
		T_\omega^{-1}(\mathcal C_{\sigma \omega})=\mathcal C_\omega, \quad \omega \in \Omega.
	\end{equation}
	Set
	\[
	\Omega_0:=\{\omega \in \Omega: \ \mu_\omega(\mathcal C_\omega)>0\}.
	\]
	By~\eqref{inv}, $\Omega_0$ is $\sigma$-invariant, and consequently $\mathbb P(\Omega_0)\in \{0, 1\}$. Without loss of generality, we may assume that $\mathbb P(\Omega_0)=1$ as if $\mathbb P(\Omega_0)=0$, we have $\mu(\mathcal C)=0$.
	Hence, we may assume that~\eqref{dec2} holds for each $\omega \in \Omega_0$. We now claim that 
	\begin{equation}\label{claim}
		\int_{\mathcal C_\omega}\phi\, d\mu_\omega=0, \quad \text{for $\omega \in \Omega_0$ and $\phi \colon M_\omega \to \mathbb R$ Lipschitz with $\int_{M_\omega}\phi \, d\mu_\omega=0$. }
	\end{equation}
	Using~\eqref{dec2} we have 
	\[
	\begin{split}
		\left |\int_{\mathcal C_\omega}\phi \, d\mu_\omega \right |=\left |\int_{M_\omega}\phi \textbf{1}_{\mathcal C_\omega}\, d\mu_\omega \right |=\left |\int_{M_\omega}\phi \cdot (\textbf{1}_{\mathcal C_{\sigma^n \omega}}\circ T_\omega^n)\, d\mu_\omega \right | &=\left |\int_{M_{\sigma^n \omega}}(L_\omega^n \phi)\cdot \textbf{1}_{\mathcal C_{\sigma^n \omega}}\, d\mu_{\sigma^n \omega}\right |\\
		&\le \|L_\omega^n \phi\|_{L^1(\mu_{\sigma^n \omega)}}\\
		&\le \|L_\omega^n \phi\|_{L^s(\mu_{\sigma^n \omega)}}\\
		&\le K(\omega)a_n(1+\|\phi\|_{\Lip}),
	\end{split}
	\]
	for $\omega \in \Omega_0$ and $n\in \mathbb N$. Letting $n\to \infty$, we obtain~\eqref{claim} (recall that $a_n\to 0$). As every continuous function on $M_\omega$ can be approximated by a Lipschitz function (in the supremum norm), we see that~\eqref{claim} also holds for $\phi\colon M_\omega \to \mathbb R$ continuous. Finally, continuous functions are dense in $L^1(\mu_\omega)$, and thus~\eqref{claim} holds for $\phi \in L^1(\mu_\omega)$. This yields $\mu_\omega(\mathcal C_\omega)=1$ for $\omega \in \Omega_0$. Thus, $\mu(\mathcal C)=1$.
	
\end{proof}

\subsection{Central limit theorem}
We first establish a quenched central limit theorem.
\begin{theorem}\label{T-CLT}
Let $\varphi \colon \mathcal M\to \mathbb R$ be a measurable map satisfying the following conditions:
\begin{itemize}
\item $\int_{M_\omega}\varphi_\omega \, d\mu_\omega=0$ for $\mathbb P$-a.e. $\omega \in \Omega$, where $\varphi_\omega:=\varphi(\omega, \cdot)$;
\item for $\mathbb P$-a.e. $\omega \in \Omega$, $\varphi_\omega$ is Lipschitz and for some $r>0$,
\begin{equation}\label{obsint}
\omega \mapsto \|\varphi_\omega\|_{\text{Lip}} \in L^r(\Omega, \mathcal F, \mathbb P).
\end{equation}
\end{itemize}
Suppose that Assumption~\ref{P assum} holds with $\mathcal I=\{s\}$ with $s\ge 1$,
\begin{equation}\label{gammaq}
\frac 1 p+\frac 1 r \le  \frac 1 2 \quad \text{and} \quad  \sum_{n=1}^\infty a_n <+\infty.
\end{equation} Then there exists $\Sigma^2\ge 0$ such that 
\begin{equation}\label{variance}
\lim_{n\to \infty}\frac 1 n \int_{M_\omega} \left (\sum_{k=0}^{n-1}\varphi_{\sigma^k \omega} \circ T_\omega^k\right )^2\, d\mu_\omega=\Sigma^2 \quad \text{for $\mathbb P$-a.e. $\omega \in \Omega$.}
\end{equation}
Moreover, for $\mathbb P$-a.e. $\omega \in \Omega$,
\[
\frac{1}{\sqrt n}\sum_{k=0}^{n-1}\varphi_{\sigma^k \omega} \circ T_\omega^k \to_d \mathcal N(0, \Sigma^2),
\]
where $\mathcal N(0, \Sigma^2)$ denotes the normal distribution with mean $0$
and variance $\Sigma^2$ provided that $\Sigma^2>0$, and the unit mass in $0$ otherwise. Finally, $\Sigma^2=0$ if and only if there exists measurable $\eta \colon \mathcal M\to \mathbb R$ that belongs to $L^2(\mu)$ such that 
\[
\varphi=\eta \circ \tau-\eta.
\]
\end{theorem}

\begin{proof}
We verify the assumptions of~\cite[Theorem 2.3]{Kifer} with $Q=\Omega$ (that is, when there is no inducing involved). 
Firstly, since $\|\varphi_\omega\|_{L^2(\mu_\omega)}\le \|\varphi_\omega\|_{\text{Lip}}$, \cite[(2.7)]{Kifer} follows readily from~\eqref{obsint} and~\eqref{gammaq} (which implies that $r\ge 2$).

Secondly, using~\eqref{dec2} and~\eqref{obsint} we have
\[
\begin{split}
\sum_{n=1}^\infty \left |\mathbb E_\omega (\varphi_\omega \cdot (\varphi_{\sigma^n \omega}\circ T_\omega^n))\right |&=\sum_{n=1}^\infty \left |\int_{M_\omega} \varphi_\omega \cdot (\varphi_{\sigma^n \omega}\circ T_\omega^n)\, d\mu_\omega \right | \\
&=\sum_{n=1}^\infty \left |\int_{M_{\sigma^n\omega}} (L_\omega^n \varphi_\omega)\varphi_{\sigma^n \omega}\, d\mu_{\sigma^n \omega}\right | \\
&\le \sum_{n=1}^\infty\|L_\omega^n \varphi_\omega \|_{L^1(\mu_{\sigma^n\omega})} \cdot \|\varphi_{\sigma^n \omega}\|_{\Lip}\\
&\le \sum_{n=1}^\infty\|L_\omega^n \varphi_\omega \|_{L^s(\mu_{\sigma^n\omega})} \cdot \|\varphi_{\sigma^n \omega}\|_{\Lip}\\
&\le K(\omega)(1+\|\varphi_\omega\|_{\Lip})\sum_{n=1}^\infty a_n  \|\varphi_{\sigma^n \omega}\|_{\Lip},
\end{split}
\]
for $\mathbb P$-a.e. $\omega \in \Omega$.  Using that $\sigma$ preserves $\mathbb P$, \eqref{gammaq} (which implies that $\frac 1 p+\frac 2 r\le 1$), \cite[Theorem 2.15]{GF} and the H\"{o}lder inequality, 
 we conclude that
\[
\mathbb E_{\mathbb P}\sum_{n=1}^\infty \left |\mathbb E_\omega (\varphi_\omega \cdot (\varphi_{\sigma^n \omega}\circ T_\omega^n))\right |<+\infty,
\]
verifying~\cite[Theorem 2.3 (i)]{Kifer}. 

Similarly, 
\[
\begin{split}
\sum_{n=1}^\infty \mathbb E_\omega |L_{\sigma^{-n}\omega}^n\varphi_{\sigma^{-n}\omega}| &=\sum_{n=1}^\infty \|L_{\sigma^{-n}\omega}^n \varphi_{\sigma^{-n}\omega}\|_{L^1(\mu_\omega)}\\
&\le \sum_{n=1}^\infty \|L_{\sigma^{-n}\omega}^n \varphi_{\sigma^{-n}\omega}\|_{L^s(\mu_\omega)}\\
&\le \sum_{n=1}^\infty K(\sigma^{-n}\omega)a_n (1+\|\varphi_{\sigma^{-n}\omega}\|_{\Lip}),
\end{split}
\]
for $\mathbb P$-a.e. $\omega \in \Omega$. Since $\sigma$ preserves $\mathbb P$ and~\eqref{gammaq} holds, we conclude that
\[
\left \|\sum_{n=1}^\infty \mathbb E_\omega |L_{\sigma^{-n}\omega}^n\varphi_{\sigma^{-n}\omega}|\right \|_{L^2(\Omega, \mathcal F, \mathbb P)}<+\infty.
\]
Hence, \cite[Theorem 2.3 (ii)']{Kifer} holds. The conclusion of the theorem now follows from~\cite[Theorem 2.3]{Kifer}.
\end{proof}

\begin{remark}
Suppose that $M_\omega=:M$ does not depend on $\omega$ and let $\psi \colon M \to \mathbb R$ be Lipschitz. Set
\[
\varphi_\omega:=\psi-\int_M \psi \, d\mu_\omega, \quad \omega \in \Omega.
\]
Clearly, $\int_M \psi \, d\mu_\omega=0$  and $\| \varphi_\omega\|_{\Lip}\le 2\|\psi\|_{\Lip}$ for $\omega \in \Omega$. Consequently, $\varphi \colon \Omega \times M \to \mathbb R$ given by $\varphi(\omega, x)=\varphi_\omega(x)$ for $(\omega, x)\in \Omega \times M$ satisfies the assumptions of the previous theorem with $r=\infty$.
\end{remark}

\subsection{Moment bounds and concentration inequalities}
Next, we establish quenched moment bounds and the corresponding concentration inequalities.
\begin{theorem}\label{MomThm}
Let $\varphi \colon \mathcal M\to \mathbb R$ be as in the statement of Theorem~\ref{T-CLT}.
Furthermore, suppose that Assumption~\ref{P assum} holds with $\mathcal I=\{s\}$ for $s\ge 2$ and
\begin{equation}\label{529c}
 \sum_{n=1}^\infty a_n<+\infty.
\end{equation}
 Then for every $\delta>0$ we have that for $\mathbb P$-a.e. $\omega \in \Omega$ and every $n\in\mathbb N$, 
$$
\|S_n^{\omega}\varphi\|_{L^s(\mu_\omega)}\leq  \bar K(\omega)n^{1/2+1/p+1/r+\delta}
$$
where $\bar K\in L^{\varrho}(\Omega,\mathcal F,\mathbb P)$ with $\frac{1}{\varrho}=\frac 1 p+\frac 1 r$ and
\[
S_n^\omega \varphi:=\sum_{k=0}^{n-1}\varphi_{\sigma^k \omega}\circ T_\omega^k.
\]
\end{theorem}

Applying the Markov inequality, we get the following result.
\begin{corollary}
Let the assumptions of Theorem~\ref{MomThm} be in force. Then
for $\mathbb P$-a.e. $\omega \in \Omega$ and every $\delta, t>0$,  we have that 
$$
\mu_{\omega}(|S_n^\omega\varphi|\geq tn)\leq t^{-s}(\bar K(\omega))^{s}n^{-s(1/2-1/p-1/r-\delta)}.
$$
\end{corollary}

\begin{proof}[Proof of Theorem \ref{MomThm}]
For $\omega \in \Omega$ and $n\in \mathbb N$, set
	\begin{equation}\label{eq:G}
	G_{\omega, n}:=\sum_{j=0}^{n-1}L_{\sigma^j \omega}^{n-j}(\varphi_{\sigma^j \omega}),
	\end{equation}
	and 
	\begin{equation}\label{mdec}
		H_{\omega, n}:=\varphi_{\sigma^n \omega}+G_{\omega, n}-G_{\omega, n+1}\circ T_{\sigma^n \omega}.
	\end{equation}
	As $G_{\omega, 0}=0$, \eqref{mdec} gives that
	\begin{equation}\label{bsums}
		S_n^\omega \varphi=\sum_{k=0}^{n-1} \varphi_{\sigma^k \omega}\circ T_\omega^k=\sum_{k=0}^{n-1} H_{\omega, k}\circ T_\omega^k+G_{\omega, n}\circ T_\omega^{n}.
	\end{equation}
	As in~\cite[Proposition 2]{DFGTV}, 
	\begin{equation}\label{mdiff}
		\mathbb E_\omega[H_{\omega, n}\circ T_\omega^n\rvert (T_\omega^{n+1})^{-1}(\mathcal B_{\sigma^{n+1}\omega})]=0, \quad \text{for $\mathbb P$-a.e. $\omega \in \Omega$ and $n\in \mathbb N_0$,}
	\end{equation}
	where $\mathbb E_\omega[\psi \rvert \mathcal F]$ denotes the conditional expectation of $\psi$ with respect to $\sigma$-algebra $\mathcal F$ and the measure $\mu_\omega$. 
	\begin{lemma}\label{GM}
		For any $\delta>0$, we have that for $\mathbb P$-a.e. $\omega \in \Omega$ and $n\in \mathbb N$,
		\[
		\|G_{\omega, n}\|_{L^s(\mu_{\sigma^n\omega})}\le A(\omega) n^{1/p+1/r+\delta} \quad \text{and} \quad \|H_{\omega, n}\|_{L^s(\mu_{\sigma^n\omega})}\le A(\omega) n^{1/p+1/r+\delta},
		\]
        where $A\in L^\varrho(\Omega, \mathcal F, \mathbb P)$ and $\frac{1}{\varrho}=\frac 1 p+\frac 1 r$.
	\end{lemma}
    \begin{proof}[Proof of the lemma]
By~\eqref{dec2} and~\eqref{obsint} we have
		\[
		\|G_{\omega, n}\|_{L^s(\mu_{\sigma^n\omega})}\le \sum_{j=0}^{n-1}\|L_{\sigma^j \omega}^{n-j}(\varphi_{\sigma^j \omega})\|_{L^s(\mu_{\sigma^n\omega})} \le  \sum_{j=0}^{n-1}K(\sigma^j \omega) (1+\|\varphi_{\sigma^j \omega}\|_{\Lip})a_{n-j},
		\]
		for $\mathbb P$-a.e. $\omega \in \Omega$ and $n\in \mathbb N$. Next, since $K\in L^p(\Omega, \mathcal F, \mathbb P)$ and~\eqref{obsint} holds,  by Lemma~\ref{CMPlemma} (applied for $\sigma^{-1}$ instead of $\sigma$), there are $K_\delta\in L^p(\Omega, \mathcal F, \mathbb P)$  and $D_\delta \in L^r(\Omega, \mathcal F, \mathbb P)$ such that 
		\[
		K(\sigma^j \omega)\le K_\delta(\omega)(j+1)^{1/p+\delta/2} \quad  \text{and} \quad 1+\|\varphi_{\sigma^j\omega}\|_{\Lip}\le D_\delta(\omega)(j+1)^{1/r+\delta/2},
\]
for $\mathbb P$-a.e. $\omega \in \Omega$ and $j\in \mathbb N_0$.
		Hence, there exist $C>0$ independent of $\omega$ and $n$ such that 
		\[
		\begin{split}
			\|G_{\omega, n}\|_{L^s(\mu_{\sigma^n\omega})}\le K_\delta(\omega)D_\delta(\omega)\sum_{j=0}^{n-1}(j+1)^{1/p+1/r+\delta}a_{n-j}&\le K_\delta (\omega) D_\delta(\omega) n^{1/p+1/r+\delta}\sum_{j=0}^{n-1}a_{n-j} \\
			&\le CK_\delta(\omega)D_\delta (\omega) n^{1/p+1/r+\delta},
		\end{split}
		\]
		for $\mathbb P$-a.e. $\omega \in \Omega$ and $n\in \mathbb N$, as $\sum_{j=1}^\infty a_j<\infty$ (see~\eqref{529c}). This establishes the first assertion of the lemma. The second assertion follows from the first as~\eqref{mdec} gives\[
		\|H_{\omega, n}\|_{L^s(\mu_{\sigma^n \omega})}\le \|\varphi_{\sigma^n \omega}\|_{\Lip}+\|G_{\omega, n}\|_{L^s(\mu_{\sigma^n \omega})}+\|G_{\omega, n+1}\|_{L^s(\mu_{\sigma^{n+1}\omega})},
		\]
and by applying Lemma~\ref{CMPlemma} to $\omega \mapsto \|\varphi_\omega\|_{\text{Lip}}\in L^r(\Omega, \mathcal F, \mathbb P)$.
\end{proof}
We now return to the proof of the theorem.
Taking into account~\eqref{bsums} and the Lemma~\ref{GM}, in order to establish the desired conclusion,  it is sufficient to obtain the upper bounds for the partial sums $S_n^\omega H:=\sum_{j=0}^{n-1}H_{\omega,j}\circ T_\omega^j$. By applying the Burkholder inequality we see that 
$$
\|S_n^\omega H\|_{L^{s}(\mu_\omega)}^2 \le C_s\sum_{j=0}^{n-1}\|H_{\omega,j}\|_{L^{s}(\mu_{\sigma^j\omega})}^2\le  C_s(A(\omega))^2\sum_{j=1}^n j^{2/p+2/r+\delta}\le C_s(A(\omega))^2 n^{1+ 2/p+2/r+\delta},
$$
for $\mathbb P$-a.e. $\omega \in \Omega$ and $n\in \mathbb N$. Here, $C_s>0$ is a constant that depends only on $s$.
\end{proof}

\subsection{Almost sure invariance principle}
Next, we have the following quenched ASIP. Our proof is based on the Skorokhod embedding theorem, and it closely follows the proof of ASIP in \cite{Su} (with certain modifications to deal with the nonuniformity in~\eqref{DEC}), but to make the paper self-contained, we will provide most of the details.
\begin{theorem}\label{T-ASIP}
Let $\varphi \colon \mathcal M \to \mathbb R$ be as in the statement of Theorem~\ref{T-CLT} where $\frac 1 p+\frac 1 r<\frac 1 4$.
Furthermore, suppose that Assumption~\ref{P assum} holds  with $\mathcal I=\{2, 2 \tilde s\}$, where
\begin{equation}\label{s12}
 \quad 2>\tilde s>\frac{1}{1-\frac 2 p-\frac 2 r},
\end{equation}
and $a_{2, n}=O(n^{-b})$  with $b>1$ satisfying
\begin{equation}\label{constantb}
(2/p+4/r)/(b-1)<1-4/r-2/p.
\end{equation}
Finally, assume that
\begin{equation}\label{1351c}
\sum_{n=1}^\infty a_{2\tilde s, n}<+\infty.
\end{equation}
Then there exists $\Sigma^2\ge 0$ such that~\eqref{variance} holds. 
Furthermore, if $\Sigma^2>0$, then for $\mathbb P$-a.e. $\omega \in \Omega$, there is a Brownian motion $B_t^\omega$, $t\ge 0$ defined on some extension of the probability space $(M_\omega, \mathcal B_\omega, \mu_\omega)$ such that
\begin{equation}\label{asipc}
\sum_{k=0}^n \varphi_{\sigma^k \omega}\circ T_\omega^k -B_{\Sigma_{\omega, n}^2}^\omega=O(n^{\frac 1 2(1-\varepsilon')}),
\end{equation}
where $\varepsilon'>0$ is sufficiently small and
\begin{equation}\label{752}
\Sigma_{\omega, n}^2:=\int_{M_\omega} \left (\sum_{k=0}^{n-1}\varphi_{\sigma^k \omega} \circ T_\omega^k\right )^2\, d\mu_\omega.
\end{equation}
\end{theorem}
We observe that the existence of $\Sigma^2\ge 0$ and~\eqref{variance} follow directly from Theorem~\ref{T-CLT}.
Following~\cite{Su}, the rest of the proof of Theorem~\ref{T-ASIP} will be divided into several parts.
		Let $G_{\omega, n}$ and $H_{\omega, n}$ be given by~\eqref{eq:G} and~\eqref{mdec}, respectively.

        \begin{lemma}\label{434}
		For $\delta>0$ and $\mathbb P$-a.e. $\omega \in \Omega$, 
		\[
		\sum_{k< n}\int_{M_\omega} H_{\omega, k}^2 \circ T_\omega^k\, d\mu_\omega =\sum_{k< n}\int_{M_{\sigma^k\omega}} H_{\omega, k}^2\, d\mu_{\sigma^k \omega}=\Sigma_{\omega, n}^2+O(n^{2/p+2/r+\delta}).
		\]
	\end{lemma}
	\begin{proof}
		Using~\eqref{bsums} and~\eqref{mdiff} we have 
		\[
		\Sigma_{\omega, n}^2=\sum_{k< n}\int_{M_\omega} H_{\omega, k}^2 \circ T_\omega^k\, d\mu_\omega+\|G_{\omega, n}\|_{L^2(\mu_{\sigma^{n}\omega})}^2
		\]
		The desired conclusion now follows easily from the Lemma~\ref{GM} (applied for $s=2$).
	\end{proof}
Set	\begin{equation}\label{756}
	\sigma_{\omega, n}^2:=\sum_{k\le  n}\int_{M_\omega} H_{\omega, k}^2 \circ T_\omega^k\, d\mu_\omega.
	\end{equation}
    The following result is a direct consequence of Lemma~\ref{434} and the requirement that $\frac 1 p +\frac 1 r <\frac 1 4$.
	\begin{lemma}\label{Sigma}
		For $\mathbb P$-a.e. $\omega \in \Omega$,
		\[
		\lim_{n\to \infty}\frac{\sigma_{\omega, n}^2}{n}=\Sigma^2.
		\]
	\end{lemma}
    \begin{lemma}\label{Lemma 4.1}
 For $\mathbb P$-a.e. $\omega \in \Omega$,
\begin{equation}\label{450-1}
\lim_{n\to \infty}\frac{\sigma_{\omega, n+1}^2}{\sigma_{\omega, n}^2}=1 \quad \text{and} \quad \lim_{n\to \infty}\frac{\int_{M_\omega} R_{\omega, n}^2\, d\mu_\omega}{\sigma_{\omega, n}^{-2}}=1,
\end{equation}
where 
\[
R_{\omega, n}:=\sum_{k\ge n}\frac{H_{\omega, k}\circ T_\omega^k}{\sigma_{\omega, k}^2}.
\]
Furthermore, $(R_{\omega, n})_n$ is a reverse martingale with respect to filtration $((T_\omega^{n})^{-1}(\mathcal B_{\sigma^n \omega}))_{n}$.
\end{lemma}

\begin{proof}
The first equality in~\eqref{450-1} follows readily from Lemma~\ref{Sigma}, while the other conclusions follow by repeating the arguments from the proof of~\cite[Lemma 4.1]{Su}.
\end{proof}

Next, the following result is a consequence of the Skorokhod embedding theorem for $(R_{\omega,n})_n$;
see \cite[Theorem 2]{SH}.

\begin{lemma}\label{Lemma 4.2}
	Let $q \in (1, \infty)$.
	For $\mathbb P$-a.e. $\omega\in\Omega$,
	there exist a constant $C_q>0$ depending only on $q$, non-increasing optional times $\tau_n^\omega\to 0$ and a Brownian motion $B_t^\omega$, $t\ge 0$ on an extended probability space of $(M_\omega, \mathcal B_\omega ,\mu_\omega)$ such that:
	\vskip0.1cm
	\begin{equation}\label{(9)}
		R_{\omega,n}=B_{\tau_n^\omega}^\omega,   
	\end{equation}
	
	\begin{equation}\label{(10)}
		\mathbb E[\tau_{n}^\omega-\tau_{n+1}^\omega|\mathcal G_{\omega,n+1}]=\mathbb E\left[\frac{H_{\omega,n}^2\circ T_\omega^n}{\sigma_{\omega,n}^4}\Bigg|T_\omega^{-(n+1)}(\mathcal B_{\sigma^{n+1}\omega})\right ],
	\end{equation}
	where $\mathcal G_{\omega,n}=\sigma\{\tau_i^\omega, \ (T_\omega^{i})^{-1}(\mathcal B_{\sigma^i \omega}): \,i\geq n\}$
	and
	\begin{equation}\label{(11)}
		C_q^{-1}\mathbb E[(\tau_{n}^\omega-\tau_{n+1}^\omega )^{q/2}|\mathcal G_{\omega,n+1}]\leq \mathbb E \left[ \biggl| \frac{H_{\omega,n} \circ T_\omega^n}{\sigma_{\omega,n}^{2} } \biggl|^q  \Bigg|T_\omega^{-(n+1)}(\mathcal B_{\sigma^{n+1}\omega})\right]\leq C_q\mathbb E[(\tau_{n}^\omega-\tau_{n+1}^\omega)^{q/2}|\mathcal G_{\omega,n+1}].
	\end{equation}
\end{lemma} 

In order to simplify the notation, in the sequel we will write $B_t$ and $\tau_n$ instead of $B_t^\omega$ and $\tau_n^\omega$, respectively. 
Next, we need the following result.
\begin{lemma}\label{Lemma 4.3}
 Let $\delta_{\omega,n}^2=\int_{M_\omega} R_{\omega,n}^2\, d\mu_ \omega$. For $\mathbb P$-a.e. $\omega \in \Omega$, the following holds: if there is $\varepsilon_0>0$ such that 
$$
\tau_n-\delta_{\omega,n}^2=o(\delta_{\omega,n}^{2+\varepsilon_0})
$$
then there is a  small $\varepsilon'>0$ such that
$$
\left|\sum_{i\leq n}H_{\omega,i}\circ T_\omega^i-\sum_{i\leq n}(B_{\delta_{\omega,i}^2}-B_{\delta_{\omega,i+1}^2})\sigma_{\omega,i}^2\right|=o(\sigma_{\omega,n}^{1-\varepsilon'}) \quad \text{$\mu_\omega$-a.e.}
$$
\end{lemma}
\begin{proof}
The proof is identical to the proof of \cite[Lemma 4.3]{Su}, but for the readers' convenience, we will provide the details. Throughout the proof, $\omega$ will belong to a full-measure subset of $\Omega$ on which the conclusions of the previous lemmas hold.
 By Lemma \ref{Lemma 4.1}, $\delta_{\omega,n}^2\approx \sigma_{\omega,n}^{-2}$. On the other hand,  Lemma~\ref{Lemma 4.2} gives that  
 $$
B_{\tau_i}=R_{\omega,i}=\sum_{k\geq i}\frac{H_{\omega,k}\circ T_\omega^k}{\sigma_{\omega,k}^2}
 $$
 $$
B_{\tau_i}-B_{\tau_{i+1}}=\frac{H_{\omega,i}\circ T_\omega^i}{\sigma_{\omega,i}^2},
 $$
 namely,
 \begin{equation}\label{M rep}
 H_{\omega,i}\circ T_\omega^i=(B_{\tau_i}-B_{\tau_{i+1}})\sigma_{\omega,i}^2.    
 \end{equation}
 For $m<n$ write 
$$
\sum_{i\leq n}H_{\omega,i}\circ T_\omega^i=\sum_{i\leq m-1}H_{\omega,i}\circ T_\omega^i+\sum_{m\leq i\leq n}H_{\omega,i}\circ T_\omega^i
=\sum_{i\leq m-1}H_{\omega,i}\circ T_\omega^i+\sum_{m\leq i\leq n}(B_{\tau_i}-B_{\tau_{i+1}})\sigma_{\omega,i}^2
$$
$$
=\sum_{i\leq m-1}H_{\omega,i}\circ T_\omega^i+B_{\delta_{\omega,m}^2}\sigma_{\omega,m}^2-B_{\delta_{\omega,n+1}^2}\sigma_{\omega,n}^2+\sum_{m+1\leq i\leq n}B_{\delta_{\omega,i}^2}(\sigma_{\omega,i}^2-\sigma_{\omega,i-1}^2)+e_{\omega,m, n},
$$
where
\[
e_{\omega, m, n}:=\sum_{m+1\le i \le n} (B_{\tau_{i}}-B_{\delta_{\omega, i}^2})(\sigma_{\omega, i}^2-\sigma_{\omega, i-1}^2)+(B_{\tau_m}-B_{\delta_{\omega, m}^2})\sigma_{\omega, m}^2-(B_{\tau_{n+1}}-B_{\delta_{\omega, n+1}^2})\sigma_{\omega, n}^2.
\]
By H\"older continuity of Brownian motion near the origin, for any $c<1/2$, fixed $m\gg 1$,
$$
|e_{\omega,m,n}| \leq \sum_{m+1\leq i\leq n}|\tau_i-\delta_{\omega,i}^2|^c(\sigma_{\omega,i}^2-\sigma_{\omega,i-1}^2)+|\tau_m-\delta_{\omega,m}^2|^c\sigma_{\omega,m}^2+|\tau_{n+1}-\delta_{\omega,n}^2|^c\sigma_{\omega,n}^2
$$
$$
\leq \sum_{m+1\leq i\leq n}o(\delta_{\omega,i}^{(2+\varepsilon_0)c})(\sigma_{\omega,i}^2-\sigma_{\omega,i-1}^2)+o(\delta_{\omega,m}^{(2+\varepsilon_0)c})\sigma_{\omega,m}^2+o(\delta_{\omega,n+1}^{(2+\varepsilon_0)c})\sigma_{\omega,n}^2.
$$
 We can choose $c<1/2$ so that $2-(2+\varepsilon_0)c<1$. Then there is small $\varepsilon'>0$ such that
$
2-(2+\varepsilon_0)c<1-\varepsilon'
$ and thus $|e_{\omega,n,m}|=o(\sigma_{\omega,n}^{1-\varepsilon'})$, a.s.
\end{proof}
\begin{lemma}\label{Lemma 13}
The conclusion of Theorem~\ref{T-ASIP} holds if for $\mathbb P$-a.e. $\omega \in \Omega$,
$$
\exists \varepsilon_0>0\text{ such that  }\,\tau_{n}-\delta_{\omega,n}^2=o(\delta_{\omega,n}^{2+\varepsilon_0}).
$$
\end{lemma}
\begin{proof}
From Lemma \ref{Lemma 4.3} and~\eqref{bsums} 
	we have that for $\mathbb P$-a.e. $\omega \in \Omega$ and $n\in \mathbb N$,
	$$
	\sum_{i\leq n}\varphi_{\sigma^i\omega}\circ T_\omega^i
	=\sum_{i\leq n}(B_{\delta_{\omega,i}^2}-B_{\delta_{\omega,i+1}^2})\sigma_{\omega,i}^2+o(\sigma_{\omega,n}^{1-\varepsilon'})+G_{\omega,n+1}\circ T_\omega^{n+1},
	$$
	where $\varepsilon'>0$ is sufficiently small. Recall from \eqref{1351c} that 
	$$
	\sum_{n=1}^\infty a_{2\tilde s, n} < \infty.
	$$
	 Therefore, Lemma~\ref{GM} applied for $s=2\tilde s$ implies 
	\[
	\begin{split}
		\int_{M_\omega}\frac{|G_{\omega, n+1}\circ T_\omega^{n+1}|^{2\tilde s}}{\sigma_{\omega, n}^{2\tilde s(1-\varepsilon')}}\, d\mu_\omega=\frac{\|G_{\omega, n+1}\|_{L^{2\tilde s}(\mu_{\sigma^{n+1}\omega})}^{2\tilde s}}{\sigma_{\omega, n}^{2\tilde s(1-\varepsilon')}} 
		&\precsim \frac{n^{2\tilde s/p+2\tilde s/r+\delta}}{n^{ \tilde s (1-\varepsilon')}}\\
		&=\frac{1}{n^{  2\tilde s (   (1 - \ve' ) / 2 - 1/p - 1/ r  ) - \delta   }},
	\end{split}
	\]
	for any $\delta>0$, $\mathbb P$-a.e. $\omega \in \Omega$ and $n\in \mathbb N$.
	Provided that $\delta$ and $\varepsilon'$ are sufficiently small, this rate is summable due to~\eqref{s12}. By the Borel-Cantelli lemma, we have 
	\begin{equation}\label{GT}
		G_{\omega, n+1}\circ T_\omega^{n+1}=o(\sigma_{\omega, n}^{1-\varepsilon'}) \quad \text{$\mu_\omega$-a.e.,}
	\end{equation}
	for $\mathbb P$-a.e. $\omega \in \Omega$.
	Hence,
	\[
	\left |\sum_{i\leq n}\varphi_{\sigma^i\omega}\circ T_\omega^i
	-\sum_{i\leq n}(B_{\delta_{\omega,i}^2}-B_{\delta_{\omega,i+1}^2})\sigma_{\omega,i}^2 \right |=o(\sigma_{\omega, n}^{1-\varepsilon'})=o(\Sigma_{\omega, n}^{1-\varepsilon'}), \quad \text{$\mu_\omega$-a.e.,}
	\]
	for $\mathbb P$-a.e. $\omega \in \Omega$.
	Since 
	\[
	\sigma_{\omega, n}^2=\sum_{i\le n}\sigma_{\omega, i}^4(\delta_{\omega, i}^2-\delta_{\omega, i+1}^2),
	\]
	it follows that 
	\[
	\left |\sum_{i\leq n}\varphi_{\sigma^i\omega}\circ T_\omega^i
	-B_{\sigma_{\omega,  n}^2}\right |=o(\sigma_{\omega, n}^{1-\varepsilon'})=o(\Sigma_{\omega, n}^{1-\varepsilon'}), \quad \text{$\mu_\omega$-a.e.,}
	\]
	for $\mathbb P$-a.e. $\omega \in \Omega$.
	Taking into account Lemma~\ref{434}, the above implies~\eqref{asipc}.
\end{proof}
Next, we decompose $\tau_{ n}-\delta_{\omega, n}^2$ as in~\cite{Su}:
\[
\tau_{\omega, n}-\delta_{\omega,n}^2=R_{\omega, n}'+R_{\omega, n}''+S_{\omega, n},
\]
where 
\[
R_{\omega, n}':=\sum_{i\ge n}\left (\tau_{\omega, i}-\tau_{\omega, i+1}-\mathbb E_\omega \left [\frac{H_{\omega, i}^2\circ T_\omega^i}{\sigma_{\omega, i}^4}\rvert ( T_\omega^{i+1})^{-1}(\mathcal B)\right ]\right ),
\]
\[
R_{\omega, n}'':=\sum_{i\ge n}\left (\mathbb E_\omega \left [\frac{H_{\omega, i}^2\circ T_\omega^i}{\sigma_{\omega, i}^4}\rvert ( T_\omega^{i+1})^{-1}(\mathcal B)\right ]-\frac{H_{\omega, i}^2\circ T_\omega^i}{\sigma_{\omega, i}^4}\right ),
\]
and 
\[
S_{\omega, n}:=\sum_{i\ge n}\left (\frac{H_{\omega, i}^2\circ T_\omega^i}{\sigma_{\omega, i}^4}-\frac{\mathbb E_{\sigma^i \omega}(H_{\omega, i}^2)}{\sigma_{\omega, i}^4}\right ).
\]
Note that $(R_{\omega, n}')_n$ and $(R_{\omega, n}'')_n$ are reverse martingales with respect to filtrations $(\mathcal G_{\omega, n})_n$ and $( (T_\omega^n)^{-1}(\mathcal B_{\sigma^n \omega}))_n$, respectively. 
\begin{lemma}\label{Lemma 14}
For $\mathbb P$-a.e. $\omega \in \Omega$ and $\varepsilon_0>0$, we have
\[
R_{\omega, n}'=o(\delta_{\omega, n}^{2+\varepsilon_0}) \quad \text{and} \quad R_{\omega, n}''=o(\delta_{\omega, n}^{2+\varepsilon_0}), \quad \text{$\mu_\omega$-a.e.}
\]
\end{lemma}
\begin{proof}
Let
\[
K_{\omega, n}:=\sum_{i\le n}\mathbb E_{\sigma^i \omega}( | H_{\omega, i} |^{2 \tilde s}).
\]
 By Lemma~\ref{GM} (for $s=2\tilde s$), \[K_{\omega, n} = O( n^{ 1 + 2\tilde s / p + 2\tilde s/ r + \delta }  ) = O( \sigma_{\omega ,n}^{ 2 + 4\tilde s/p + 4\tilde s / r + \delta } )
,\] 
for any $\delta > 0$. By the martingale maximal inequality,
\begin{align*}
    	\mathbb E_\omega \left (\frac{\sup_{i\ge n}|R_{\omega, i}'|}{\delta_{\omega, n}^{2+\varepsilon_0}}\right )^{\tilde s} \precsim
	\frac{\mathbb E_\omega ( |R_{\omega, n}'|^{\tilde s} )}{\delta_{\omega, n}^{\tilde s(2+\varepsilon_0)}}.
\end{align*}
Let
$$
D_{\omega, i} 
= \tau_{ i}-\tau_{ i+1}-\mathbb E_\omega \left [\frac{H_{\omega, i}^2\circ T_\omega^i}{\sigma_{\omega, i}^4}\rvert ( T_\omega^{i+1})^{-1}(\mathcal B_{\sigma^{i+1}\omega})\right ].
$$
The von Bahr--Esseen type inequality for (reverse) martingales in \cite[Theorem 2.13]{MPU} combined with a standard truncation argument implies
\begin{align*}
	\bE_\omega[ | R_{\omega, n}' |^{\tilde s}  ] \precsim  \sum_{i \ge n}   \bE_\omega[  | D_{\omega ,i} |^{\tilde s}  ].
\end{align*}
 On the other hand, Lemma \ref{Lemma 4.2} implies
\begin{align*}
	\bE_\omega [  | D_{ \omega, n } |^{\tilde s} ] \precsim \bE_\omega \biggl[   \biggl| 
	\frac{ H_{\omega, n} \circ T_\omega^n }{ \sigma_{ \omega, n }^2 }
	\biggr|^{2 \tilde s}   \biggr].
\end{align*}
Therefore, we obtain
\[
\begin{split}
		\mathbb E_\omega \left (\frac{\sup_{i\ge n}|R_{\omega, i}'|}{\delta_{\omega, n}^{2+\varepsilon_0}}\right )^{\tilde s} &\precsim \frac{\mathbb E_\omega ( | R_{\omega, n}' |^{\tilde s} )}{\delta_{\omega, n}^{\tilde s(2+\varepsilon_0)}} \precsim \frac{1}{\delta_{\omega, n}^{\tilde s(2+\varepsilon_0)}}\sum_{i\ge n}\frac{\mathbb E_{\sigma^i \omega}( |H_{\omega, i}|^{2\tilde s})}{\sigma_{\omega, i}^{4\tilde s}}
        		=\frac{1}{\delta_{\omega, n}^{\tilde s(2+\varepsilon_0)}}\sum_{i\ge n}\frac{K_{\omega, i}-K_{\omega, i-1}}{\sigma_{\omega, i}^{4\tilde s}}\\
                		&=\frac{1}{\delta_{\omega, n}^{\tilde s(2+\varepsilon_0)}} \cdot \left (-\frac{K_{\omega, n-1}}{\sigma_{\omega, n}^{4\tilde s}}+\sum_{i\ge n}
		K_{\omega, i} \left (\frac{1}{\sigma_{\omega, i}^{4\tilde s}}-\frac{1}{\sigma_{\omega, i+1}^{4\tilde s}}\right )\right )\\
        		&\precsim \frac{1}{\delta_{\omega, n}^{\tilde s(2+\varepsilon_0)}} \cdot \left (\frac{K_{\omega, n-1}}{\sigma_{\omega, n}^{4\tilde s}}+\sum_{i\ge n}
 		K_{\omega, i} \left (\frac{1}{\sigma_{\omega, i}^{4\tilde s}}-\frac{1}{\sigma_{\omega, i+1}^{4\tilde s}}\right )\right )\\
        		&\precsim \frac{1}{\delta_{\omega, n}^{\tilde s(2+\varepsilon_0)}} \cdot \left (\frac{1}{\sigma_{\omega, n}^{4\tilde s - 2 - 4\tilde s/ p - 4\tilde s/ r - \delta }}+\sum_{i\ge n}\frac{\sigma_{\omega, i+1}^{4\tilde s}-\sigma_{\omega, i}^{4\tilde s}}{\sigma_{\omega, i}^{8\tilde s - 2 - 4\tilde s / p - 4\tilde s / r -\delta}}\right ) \\
                		&\precsim \frac{1}{\delta_{\omega, n}^{\tilde s(2+\varepsilon_0)}} \cdot \left (\frac{1}{\sigma_{\omega, n}^{4\tilde s - 2 - 4\tilde s/p - 4\tilde s /r -\delta}}+\int_{\sigma_{\omega, n}^{4\tilde s}}^\infty  x^{- \frac{8\tilde s - 2 - 4\tilde s/p - 4\tilde s/r -\delta}{4\tilde s}}\, dx\right ) .\\
\end{split}
\]
	By \eqref{s12}, 
$$
	\frac{8\tilde s - 2 - 4\tilde s/p - 4\tilde s/r }{4\tilde s} = 2 - \frac{1}{2\tilde s} - \frac{1}{p} - \frac1r >1.
$$
Therefore,
\begin{align*}
	\int_{\sigma_{\omega, n}^{4\tilde s}}^\infty  x^{- \frac{8\tilde s - 2 - 4\tilde s/p - 4\tilde s/r -\delta}{4\tilde s}}\, dx
    	= \int_{\sigma_{\omega, n}^{4\tilde s}}^\infty  x^{  - 2 + \frac{1}{2\tilde s} + \frac{1}{p} + \frac1r + \frac{\delta}{4\tilde s} }\, dx
        	\precsim  \frac{1}{ \sigma_{\omega, n}^{ 4\tilde s - 2 - 4\tilde s / p - 4\tilde s / r  -  \delta  } }.
\end{align*}
	Recalling from Lemma \ref{Lemma 4.1} that $\delta_{\omega,n}^2 \asymp \sigma_{\omega,n}^{-2} \asymp n^{-1}$, we conclude that 
\begin{equation}\label{450}
		\mathbb E_\omega \left (\frac{\sup_{i\ge n}|R_{\omega, i}'|}{\delta_{\omega, n}^{2+\varepsilon_0}}\right )^{\tilde s} \precsim \sigma_{\omega, n}^{ -2\tilde s + 2 + 4\tilde s / p + 4\tilde s / r + \delta + \tilde s \ve_0  } \precsim n^{  -\tilde s + 1 + 2\tilde s/p + 2\tilde s / r + \delta/2 + \tilde s\ve_0 / 2  }
\end{equation}
Due to~\eqref{s12}, for sufficiently small $\ve_0$ and $\delta$ there exists $w > 0$ such that
\[
	w(  \tilde s - 1 - 2\tilde s/p - 2\tilde s / r - \delta/2 - \tilde s\ve_0 / 2  ) > 1.
\]
	It follows by the Borel--Cantelli lemma that 
\[
	\sup_{i\ge \lfloor N^w\rfloor }|R_{\omega, i}'|=o(\delta_{\omega, \lfloor N^w\rfloor}^{2+\varepsilon_0}), \quad \text{$\mu_\omega$-a.e.}
\]
	For any $n\in \mathbb N$, there is $N\in \mathbb N_0$ such that $\lfloor N^w\rfloor \le n<\lfloor (N+1)^w\rfloor $, and thus 
    \[
    	\frac{|R_{\omega, n}'|}{\delta_{\omega, n}^{2+\varepsilon_0}}\le \frac{\sup_{i\ge \lfloor N^w\rfloor}|R_{\omega, i}'|}{\delta_{\omega, n}^{2+\varepsilon_0}}=\frac{\sup_{i\ge \lfloor N^w\rfloor}|R_{\omega, i}'|}{\delta_{\omega, \lfloor N^w\rfloor}^{2+\varepsilon_0}}\cdot \frac{\delta_{\omega, \lfloor N^w\rfloor}^{2+\varepsilon_0}}{\delta_{\omega, n}^{2+\varepsilon_0}}=o(1)\cdot \frac{\delta_{\omega, \lfloor N^w\rfloor}^{2+\varepsilon_0}}{\delta_{\omega, n}^{2+\varepsilon_0}}.
    \]
    	Since $\delta_{\omega, n}^2\asymp n^{-1}$, we have $\frac{\delta_{\omega, \lfloor N^w\rfloor}^{2+\varepsilon_0}}{\delta_{\omega, n}^{2+\varepsilon_0}}=O(1)$ which yields the desired conclusion. The estimate for $R_{\omega, n}''$ is similar.
\end{proof}
Let
\begin{equation}\label{825}
S_{\omega, n}':=\sum_{i\le n}\left (H_{\omega, i}^2\circ T_\omega^i-\mathbb E_{\sigma^i \omega}(H_{\omega, i}^2)\right ).
\end{equation}
The proof of the following lemma is identical to the proof of~\cite[Lemma 4.6]{Su}.
\begin{lemma}\label{Lemma eps} For $\mathbb P$-a.e. $\omega \in \Omega$ the following holds:
if there is $\varepsilon'>0$ such that
\[
S_{\omega, n}'=o(\sigma_{\omega, n}^{2(1-\epsilon')}),
\]
then there is $\varepsilon_0>0$ such that
\[
S_{\omega, n}=o(\delta_{\omega, n}^{2+\epsilon_0}).
\]
\end{lemma}
Following~\cite{Su}, we decompose $S_{\omega, n}'$ as a sum of the following terms: 
\begin{equation}\label{term1}
    \sum_{i\le n} \left ( \varphi_{\sigma^i \omega}^2 \circ T_\omega^i-\mathbb E_{\sigma^i \omega}(\varphi_{\sigma^i \omega}^2)\right ),
\end{equation}
\begin{equation}\label{term2}
\mathbb E_{\sigma^{n+1}\omega}(G_{\omega, n+1}^2),
\end{equation}
\begin{equation}\label{term3}
-G_{\omega, n+1}^2\circ T_\omega^{n+1},
\end{equation}
\begin{equation}\label{term4}
-2\sum_{i\le n} H_{\omega, i} \circ T_\omega^i \cdot G_{\omega, i+1} \circ T_\omega^{i+1},
\end{equation}
and 
\begin{equation}\label{term5}
2\sum_{i\le n}\left (\varphi_{\sigma^i \omega} \circ T_\omega^i \cdot G_{\omega, i} \circ T_\omega^i-\mathbb E_{\sigma^i \omega}(\varphi_{\sigma^i \omega} G_{\omega, i})\right ).
\end{equation}
In the sequel, $\omega$ will belong to a full-measure subset of $\Omega$ on which the conclusions of the previous lemmas hold.

To handle the term \eqref{term1}, we define $\phi_\omega=\varphi_\omega^2$. Then $\|\phi_\omega\|_{\Lip}\leq 3\|\varphi_\omega\|_{\Lip}^2\in L^{r/2}(\Omega,\mathcal F,\mathbb P)$. Thus by applying Theorem \ref{MomThm} we see that for every $\delta>0$
$$
\|S_n^\omega\phi-\mu_\omega(S_n^\omega\phi)\|_{L^2(\mu_\omega)}\leq \tilde K(\omega)n^{1/2+1/p+2/r+\delta}
$$
where $\tilde K\in L^{\tilde p}(\Omega,\mathcal F,\mathbb P)$ and $\tilde p>0$ is defined by $1/\tilde p=1/p+2/r$. Using that $\tilde K(\sigma^j\omega)=o(j^{1/\tilde p})$ (by the mean ergodic theorem) and applying Lemma~\ref{proclemma} we conclude that 
$$
\eqref{term1}=O(n^{1/2+1/\tilde p}(\ln n)^{3/2+\delta})=O(n^{1/2+1/p+2/r}(\ln n)^{3/2+\delta})=o(\sigma_{\omega, n}^{2(1-\varepsilon')}),
$$
if $\varepsilon'$ is sufficiently small.

Next, by Lemma~\ref{GM} (applied for $s=2$) for any $\delta>0$,
\[
\eqref{term2}=\|G_{\omega, n+1}\|_{L^2(\mu_{\sigma^{n+1}\omega})}^2
=O(n^{2/p+2/r+\delta})=o(\sigma_{\omega, n}^{2(1-\varepsilon')}),
\]
if $\varepsilon'$ is sufficiently small. In addition, by~\eqref{GT} we have that 
\[
\eqref{term3}=o(\sigma_{\omega, n}^{2(1-\varepsilon')}) \quad \text{$\mu_\omega$-a.e.}
\]
We now turn to~\eqref{term4}. We first note that $(H_{\omega, n}\circ T_\omega^n\cdot G_{\omega, n+1}\circ T_\omega^{n+1})_n$ is a reverse martingale difference with respect to filtration $( (T_\omega^n)^{-1}(\mathcal B_{\sigma^n \omega}))_n$. Hence, using the von Bahr--Esseen type inequality for (reverse) martingales in \cite[Theorem 2.13]{MPU} together with H\"older's inequality, 
we obtain
\begin{equation}\label{1226}
\begin{split}
		\int_{M_\omega}\left | \frac{\sum_{i\le n} H_{\omega, i} \circ T_\omega^i \cdot G_{\omega, i+1} \circ T_\omega^{i+1}}{\sigma_{\omega, n}^{2(1-\varepsilon')}}\right |^{\tilde s}\, d\mu_\omega  
        		&\precsim \frac{\sum_{i\le n}\int_{M_\omega} | H_{\omega, i} \circ T_\omega^i \cdot  G_{\omega, i+1} \circ T_\omega^{i+1} |^{\tilde s} \, d\mu_\omega}{\sigma_{\omega, n}^{2\tilde s(1-\varepsilon')}} \\
                		&\le \sum_{i\le n}\frac{\|H_{\omega, i}^{\tilde s}\|_{L^2(\mu_{\sigma^i \omega})}\cdot \|G_{\omega, i+1}^{\tilde s}\|_{L^2(\mu_{\sigma^{i+1}\omega})}}{\sigma_{\omega, n}^{2\tilde s(1-\varepsilon')}} \\
                        		&=\sum_{i\le n}\frac{\|H_{\omega, i}\|_{L^{2\tilde s}(\mu_{\sigma^i \omega})}^{\tilde s}\cdot \|G_{\omega, i+1}\|_{L^{2\tilde s}(\mu_{\sigma^{i+1}\omega})}^{\tilde s}}{\sigma_{\omega, n}^{2\tilde (1-\varepsilon')}}.
\end{split}
\end{equation}
Now,  Lemma~\ref{GM} implies 
\begin{align*}
	\|H_{\omega, i}\|_{L^{2\tilde s}(\mu_{\sigma^i \omega})}^{\tilde s}\precsim i^{ \tilde s/p + \tilde s/r + \delta \tilde s } 
	\quad \text{and} \quad 
		\|G_{\omega, i}\|_{L^{2\delta s}(\mu_{\sigma^i \omega})}^{\tilde s} \precsim i^{ \tilde s/p + \tilde s/r + \delta \tilde s } 
\end{align*}
for any $\delta > 0$. Consequently,
\begin{align}\label{1226}
	\int_{M_\omega}\left | \frac{\sum_{i\le n} H_{\omega, i} \circ T_\omega^i \cdot G_{\omega, i+1} \circ T_\omega^{i+1}}{\sigma_{\omega, n}^{2(1-\varepsilon')}}\right |^{\tilde s}\, d\mu_\omega  
	\precsim n^{ -\tilde s + 1 + 2\tilde s/p + 2\tilde s/r + 2\tilde s \delta + \tilde s \ve'  }.
\end{align}
for any $\delta>0$.  Further we choose $w > 0$ such that $w(  \tilde s - 1 - 2\tilde s/p - 2\tilde s/r  ) > 1$. Then if $\delta$ and $\ve'$ are sufficiently small, it follows from~\eqref{1226} and the Borel-Cantelli lemma that 
\[
\sum_{i\le \lfloor N^w\rfloor} H_{\omega, i}\circ T_\omega^i \cdot G_{\omega, i+1}\circ T_\omega^{i+1}=o(\sigma_{\omega, \lfloor N^w\rfloor}^{2(1-\varepsilon')}) \quad \text{$\mu_\omega$-a.e.}
\]
For any $n\in \mathbb N$, there is $N\in \mathbb N_0$ such that $\lfloor N^w\rfloor \le n< \lfloor (N+1)^w\rfloor$. Using Doob's martingale inequality, \cite[Theorem 2.13]{MPU},  and Lemma~\ref{GM},
\[
\begin{split}
	&\int_{M_\omega} \frac{\max_{\lfloor N^w \rfloor \le j\le \lfloor (N+1)^w\rfloor} |\sum_{j\le i \le \lfloor (N+1)^w\rfloor}H_{\omega, i}\circ T_\omega^i \cdot G_{\omega, i+1}\circ T_\omega^{i+1} |^{\tilde s}}{\sigma_{n, \lfloor N^w\rfloor}^{2\tilde s(1-\varepsilon')}}\, d\mu_\omega  \\
    	&\precsim  \frac{\int_{M_\omega} |\sum_{\lfloor N^w\rfloor\le i \le \lfloor (N+1)^w\rfloor}H_{\omega, i}\circ T_\omega^i \cdot G_{\omega, i+1}\circ T_\omega^{i+1}|^{\tilde s}\, d\mu_\omega}{\sigma_{n, \lfloor N^w\rfloor}^{2\tilde s(1-\varepsilon')}}  \\
        	&\precsim  \frac{\int_{M_\omega} \sum_{\lfloor N^w\rfloor\le i \le \lfloor (N+1)^w\rfloor} |H_{\omega, i}\circ T_\omega^i \cdot G_{\omega, i+1}\circ T_\omega^{i+1}|^{\tilde s}\, d\mu_\omega}{\sigma_{n, \lfloor N^w\rfloor}^{2\tilde s(1-\varepsilon')}}  \\
            	&\le \frac{\sum_{\lfloor N^w\rfloor\le i\le \lfloor (N+1)^w\rfloor}\|H_{\omega, i}\|_{L^{2\tilde s}(\mu_{\sigma^i \omega})}^{\tilde s}\cdot \|G_{\omega, i+1}\|_{L^{2\tilde s}(\mu_{\sigma^{i+1}\omega})}^{\tilde s}}{\sigma_{n, \lfloor N^w\rfloor}^{2\tilde s(1-\varepsilon')}} \\
                	&\precsim \frac{\sum_{\lfloor N^w\rfloor\le i\le \lfloor (N+1)^w\rfloor}i^{2\tilde s/p+2\tilde s/r+ 2\tilde s\delta}}{\sigma_{n, \lfloor N^w\rfloor}^{2\tilde s(1-\varepsilon')}}\precsim \frac{N^{w(2\tilde s/p+2\tilde s/r+2\tilde s\delta)}N^{w-1}}{N^{w\tilde s(1-\varepsilon')}}=\frac{1}{N^{w \tilde s (1-\varepsilon')-w(2\tilde s/p+2\tilde s/r+2\tilde s\delta)-w+1}}.
\end{split}
\]
Note that our choice of $s$ and $w$ implies that the last term above is summable (provided that $\varepsilon'$ and $\delta$ are sufficiently small). By the Borel-Cantelli lemma, 
\[
\max_{\lfloor N^w \rfloor \le j\le \lfloor (N+1)^w\rfloor}  \left |\sum_{j\le i \le \lfloor (N+1)^w\rfloor}H_{\omega, i}\circ T_\omega^i \cdot G_{\omega, i+1}\circ T_\omega^{i+1} \right |=o(\sigma_{\omega, \lfloor N^w\rfloor}^{2(1-\varepsilon')}) \quad \text{$\mu_\omega$-a.e.}
\]
We now have
\[
\begin{split}
	\left |\sum_{i\le n}H_{\omega, i}\circ T_\omega^i \cdot G_{\omega, i+1}\circ T_\omega^{i+1} \right | &\le \left |\sum_{i\le \lfloor (N+1)^w\rfloor}H_{\omega, i}\circ T_\omega^i \cdot G_{\omega, i+1}\circ T_\omega^{i+1} \right | \\
    	&\phantom{\le}+\max_{\lfloor N^w \rfloor \le j\le \lfloor (N+1)^w\rfloor} \left  |\sum_{j\le i \le \lfloor (N+1)^w\rfloor}H_{\omega, i}\circ T_\omega^i \cdot G_{\omega, i+1}\circ T_\omega^{i+1} \right | \\
	&\le o(\sigma_{\omega, \lfloor (N+1)^w\rfloor}^{2(1-\varepsilon')})+o(\sigma_{\omega, \lfloor N^w\rfloor}^{2(1-\varepsilon')})=o(\sigma_{\omega, \lfloor N^w\rfloor}^{2(1-\varepsilon')})\le o(\sigma_{\omega, n}^{2(1-\varepsilon')}),
\end{split}
\]
$\mu_\omega$-a.e.

Finally, it remains to deal with~\eqref{term5}. 
Set
\[
\begin{split}
U_{\omega, n}&:=\sum_{i\le n}\left (\varphi_{\sigma^i \omega} \circ T_\omega^i \cdot G_{\omega, i} \circ T_\omega^i-\mathbb E_{\sigma^i \omega}(\varphi_{\sigma^i \omega} G_{\omega, i})\right )\\
&=\sum_{i\le n} \left ( (\varphi_{\sigma^i \omega}\cdot G_{\omega, i})\circ T_\omega^i-\mathbb E_{\sigma^i \omega}(\varphi_{\sigma^i \omega}G_{\omega, i})\right ).
\end{split}
\]
We now aim to estimate $\int_{M_\omega} |U_{\omega, n}-U_{\omega, m}|^2\, d\mu_\omega$ for $m<n$. To this end, we start by noticing that
\[
\begin{split}
&\int_{M_\omega} |U_{\omega, n}-U_{\omega, m}|^2\, d\mu_\omega \\
&=\sum_{m< i \le n}\int_{M_\omega} \left ( (\varphi_{\sigma^i \omega}\cdot G_{\omega, i})\circ T_\omega^i-\mathbb E_{\sigma^i \omega}(\varphi_{\sigma^i \omega}G_{\omega, i})\right )^2\, d\mu_\omega \\
&\phantom{=}+2\sum_{m< j\le n}\sum_{m< i \le j-1} \int_{M_\omega} ((\varphi_{\sigma^j \omega}\cdot G_{\omega, j})\circ T_\omega^j-\mathbb E_{\sigma^j \omega}(\varphi_{\sigma^j \omega}G_{\omega, j}))((\varphi_{\sigma^i \omega}\cdot G_{\omega, i})\circ T_\omega^i-\mathbb E_{\sigma^i \omega}(\varphi_{\sigma^i \omega}G_{\omega, i}))\, d\mu_\omega \\
&=:(I)_{\omega, m, n}+(II)_{\omega, m, n}.
\end{split}
\]
Moreover, we have 
\[
\begin{split}
(I)_{\omega, m, n}\le \sum_{m< i \le n}\int_{M_\omega} (\varphi_{\sigma^i \omega}G_{\omega, i})^2\circ T_\omega^i\, d\mu_\omega &=\sum_{m< i \le n}\mathbb E_{\sigma^i \omega}(\varphi_{\sigma^i \omega}G_{\omega, i})^2\\
&\le \sum_{m< i \le n} \|\varphi_{\sigma^i \omega}\|_{\Lip}^2 \cdot \|G_{\omega, i}\|_{L^2(\mu_{\sigma^i \omega})}^2.
\end{split}
\]
Together with Lemma~\ref{GM} this gives that
\[
(I)_{\omega, m, n}\precsim n^{2/p+4/r+\delta}(n-m),
\]
for any $\delta >0$. We now focus on $(II)_{\omega, m, n}$ (ignoring the factor $2$). Writing
\[
\psi_{\omega, i}=\varphi_{\sigma^i \omega}G_{\omega, i}-\mathbb E_{\sigma^i \omega}(\varphi_{\sigma^i \omega}G_{\omega, i})
\]
we have 
\[
\begin{split}
(II)_{\omega, m, n} &=\sum_{m< j\le n}\sum_{m< i \le j-1}\int_{M_\omega} \psi_{\omega, j}\circ T_\omega^j \cdot \psi_{\omega, i}\circ T_\omega^i \, d\mu_\omega  \\
&=\sum_{m<j\le n}\sum_{m< i \le j-1}\int_{M_{\sigma^i \omega}} \psi_{\omega, j}\circ T_{\sigma^i \omega}^{j-i}\cdot \psi_{\omega, i}\, d\mu_{\sigma^i \omega} \\
&=\sum_{m< j\le n}\sum_{m< i \le j-1}\int_{M_{\sigma^j \omega}} \psi_{\omega, j}L_{\sigma^i \omega}^{j-i}(\psi_{\omega, i})\, d\mu_{\sigma^j \omega}.
\end{split}
\]
Let $\delta_* > 0$. Following \cite[Lemma 3.4]{NTV}, we decompose
\begin{align*}
	&\sum_{m< i \le j-1}\int_{M_{\sigma^j\omega}} \psi_{\omega, j}L_{\sigma^i \omega}^{j-i}(\psi_{\omega, i})\, d\mu_{\sigma^j \omega} \\
	&= \sum_{m< i \le j-j^{\delta_*}}\int_{M_{\sigma^j\omega}} \psi_{\omega, j}L_{\sigma^i \omega}^{j-i}(\psi_{\omega, i})\, d\mu_{\sigma^j \omega} 
	+ \sum_{ m \vee ( j-j^{\delta_* )} < i \le j-1}\int_{M_{\sigma^j \omega}} \psi_{\omega, j}L_{\sigma^i \omega}^{j-i}(\psi_{\omega, i})\, d\mu_{\sigma^j \omega} \\
	&= \sum_{m< i \le j-j^{\delta_*}}\int_{M_{\sigma^j\omega}} \psi_{\omega, j}L_{\sigma^i \omega}^{j-i}(\psi_{\omega, i})\, d\mu_{\sigma^j \omega} 
	+ O( j^{\delta_*}  j^{ \frac2p+\frac{4}{r} +\delta} ),
\end{align*}
for any $\delta >0$.
The last equality follows from
Lemma \ref{GM}, which implies
$$
\Vert \psi_{\omega, j} \Vert_{ L^2( \mu_{\sigma^j \omega} ) }\precsim \|\varphi_{\sigma^j\omega}\|_{\Lip}\cdot \| G_{\omega, j}\|_{L^2(\mu_{\sigma^j \omega})}\precsim j^{1/p+2/r+\delta}.
$$

To deal with the remaining term, we write
$$
L_{\sigma^i\omega}^{j-i} ( \psi_{\omega, i} )
= \sum_{q = 0}^{i-1}
\bigl[  L_{  \sigma^i \omega }^{j-i} \bigl(  L_{ \sigma^q \omega }^{i-q}   ( \varphi_{\sigma^q \omega} )  \varphi_{ \sigma^i \omega }  \bigr)  \bigr]_{ \sigma^j \omega },
$$
and apply \eqref{DEC} together with Lemma~\ref{CMPlemma} 
to obtain 
\begin{align*}
	\Vert L_{\sigma^i \omega}^{j-i}(\psi_{\omega, i}) \Vert_{ L^2( \mu_{\sigma^j \omega} )  }	
	\precsim ( i + 1)^{ 1 + 1 / p +2/r+ \delta } a_{2, j-i}
\end{align*}
for any $\delta > 0$. Thus, 
\begin{align*}
	&\left |\sum_{m< i \le j-j^{\delta_*}}\int_{M_{\sigma^j \omega}} \psi_{\omega, j}L_{\sigma^i \omega}^{j-i}(\psi_{\omega, i})\, d\mu_{\sigma^j \omega} \right |\le 
	\sum_{m< i \le j-j^{\delta_*}} \Vert \psi_{\omega, j} \Vert_{ L^2( \mu_{\sigma^j\omega} ) } 
	\Vert L_{\sigma^i \omega}^{j-i}(\psi_{\omega, i}) \Vert_{ L^2( \mu_{\sigma^j\omega} ) }
	\\ 
	&\precsim  j^{1/p+2/r+\delta} \sum_{m< i \le j-j^{\delta_*}}
	( i + 1)^{ 1 + 1 / p +2/r+ \delta } a_{2, j-i}
	\precsim j^{ 1 + \frac2p +\frac{4}{r} +\delta + \delta_*( 1-b)  }
\end{align*}
where we recall that $a_{2, k}=O(k^{-b})$.  It follows that 
$$
(II)_{\omega, m, n} \precsim \sum_{m< j \le n} 
(j^{\delta_*+ \frac2p +\frac 4 r+ \delta }+j^{1+2/p+4/r+\delta_*(1-b)}),
$$
for any $\delta > 0$. In particular,
\begin{equation}\label{eq:U_nm}
	\int_{M_\omega} |U_{\omega, n}-U_{\omega, m}|^2\, d\mu_\omega \precsim
	 \sum_{m< j \le n}(j^{\delta_*+ 2/p +4/r+ \delta}+j^{1+2/p+4/r+\delta_*(1-b)}).
\end{equation}
Using Lemma~\ref{proclemma} we conclude that, almost surely, 
$$
|U_{\omega, n}|=O(n^{\frac{\delta_*+2/p+4/r+\delta+1}{2}}+n^{\frac{1+2/p+4/r+\delta+1+\delta_*(1-b)}{2}}).
$$
Therefore, $U_{\omega,n}=o(n^\zeta)$ for $\zeta<1$ if we can choose $\delta_*$ such that
\begin{equation}\label{deltastar}
\delta_*+2/p+4/r+\delta+1<2, \,\,2+2/p+4/r+\delta+\delta_*(1-b)<2
\end{equation}
namely, assuming that $b>1$, if 
$$
(2/p+4/r+\delta)/(b-1)<1-4/r-2/p-\delta,
$$
which follows from~\eqref{constantb}.


\subsection{CLT rates}
We use the same notation as in the previous subsection.
Let us begin with the following standard result.
\begin{proposition}\label{prop918}
Let the conditions of Theorem \ref{MomThm} be in force with $2s$ instead of $s$. Suppose that for $\mathbb P$-a.e. $\omega \in \Omega$ and all $\delta>0$,
 $$
\left\|\sum_{j=0}^{n-1}(H_{\omega,j}\circ T_\omega^j)^2-\sum_{j=0}^{n-1}\mathbb E_{\sigma^j \omega}[(H_{\omega,j})^2]\right\|_{L^s(\mu_\omega)}^s=O(n^{s/2+2s/p+2s/r+A+\delta})
 $$
 for some $A\geq 1$.
Let $\Phi(t)$ be the standard normal distribution function. Then for every $\delta>0$,
 $$
\sup_{t\in\mathbb R}|\mu_\omega(S_{n}^\omega\varphi\leq t\Sigma_{\omega,n})-\Phi(t)|=O(n^{-\frac{s}{2(2s+1)}+\frac{2s}{p(2s+1)}+\frac{2s}{r(2s+1)}+\frac{A}{2s+1}+\delta}),
$$    
where $\Sigma_{\omega, n}^2$ is given by~\eqref{752}.
\end{proposition}

\begin{proof}
Throughout the proof $C(\omega)$ will denote a generic constant independent of $n$ and $t$. We will also use the same notation as in the proof of Theorem~\ref{MomThm}.

 By applying \footnote{Note that one can replace there $\mathbb E[H_{\omega,i}^2|(T_\omega)^{-(i+1)}\mathcal B]$ with $H_{\omega,i}^2$ since their difference is a martingale difference} \cite[Theorem 1]{Haus} with $\delta=s-1$ and $X_k=\sigma_{\omega,n}^{-1}H_{\sigma^k\omega}\circ T_\omega^k, k<n$ and using that $\|H_{\omega,n}\|_{L^{2s}(\mu_{\sigma^n \omega})}=O(n^{1/p+1/r+
 \delta})$ (see Lemma~\ref{GM}) we have that for $\mathbb P$-a.e. $\omega \in \Omega$,
 $$ 
\sup_{t\in\mathbb R}|\mu_\omega(S_{n}^\omega H\leq t\|S_n^\omega H\|_{L^2(\mu_\omega)})-\Phi(t)|\leq C(\omega)\left(n^{-s+2s/p+2s/r+\delta s}+n^{-s/2+2s/p+2s/r+A+\delta }\right)^{\frac{1}{2s+1}},
$$       
where $C(\omega)>0$. Here we used that $3+2\delta=2s+1$, $2+\delta=s+1$ and $\sigma_{\omega,n}\asymp\Sigma n^{1/2}$.
Next, the rates for the sum $S_n^\omega\varphi$ follow from the above rates for the sum $S_n^\omega H$. In fact, by \eqref{mdec} and Lemma~\ref{GM} we have $\|S_n^\omega\varphi-S_n^\omega H\|_{L^{2s}(\mu_\omega)}=O(n^{1/p+1/r+\delta})$. Therefore,
$$
\|S_{n}^\omega \varphi/\Sigma_{\omega,n}-S_{n}^\omega H/\sigma_{\omega,n}\|_{L^s(\mu_\omega)}\leq C(\omega)n^{1/p+1/r+\delta},
$$
for $\mathbb P$-a.e. $\omega \in \Omega$ and $n\in \mathbb N$, where $\sigma_{\omega, n}^2$ is given by~\eqref{756}.
Now, the CLT rates for $S_n^\omega\varphi$ follow from the rates for $S_n^\omega H$ together with \cite[Lemma 3.3]{SPA} applied with $a=2s$.
\end{proof}

Next, let us show that the conditions of the previous proposition hold.
 \begin{proposition}\label{QuadProp}
Let the assumptions of Theorem \ref{T-ASIP} be in force.  Then for $\mathbb P$-a.e. $\omega \in \Omega$ and all $\delta>0$,
 $$
\left\|\sum_{j=0}^{n-1}(H_{\omega,j}\circ T_\omega^j)^2-\sum_{j=0}^{n-1}\mathbb E_{\sigma^j \omega}[(H_{\omega,j})^2]\right\|_{L^{\tilde s}(\mu_\omega)}
$$
$$
=O(n^{\frac12+\frac{2}p+\frac{2}{r}+\delta}+n^{ -1 + 2/\tilde s + 2/p + 2/r + 2\delta}+n^{1/2+\delta_*/2+ 1/p +2/r+ \delta}+n^{1+1/p+2/r+\delta_*(1-b)/2}),
 $$     
 where $\delta_*>0$ satisfies~\eqref{deltastar}.
 \end{proposition}
 \begin{proof}
 Observe that 
 \[
 \sum_{j=0}^{n-1}(H_{\omega,j}\circ T_\omega^j)^2-\sum_{j=0}^{n-1}\mathbb E_{\sigma^j \omega}[(H_{\omega,j})^2]
 \]
 coincides with $S_{\omega, n}'$ introduced in the proof of Theorem~\ref{T-ASIP} (see~\eqref{825}).
 Recall the decomposition of $S_{\omega,n}'$ into the sum of terms in \eqref{term1}--\eqref{term5}.  Let us denote the corresponding terms by $I_1,I_2,I_3,I_4,I_5$, respectively. 
Next, notice that
$$
\mathbb E_\omega [I_1^2]\leq C(\omega)n^{1+2/r+1/p}
$$
for some random variable $C(\omega)>0$ that does not depend on $n$. 
Indeed, this holds since by \eqref{obsint} and the mean ergodic theorem we have $\|\varphi_{\sigma^j\omega}\|_{\Lip}=o(j^{1/r})$ and $K(\sigma^j\omega)=O(j^{1/p})$ since $K\in L^p(\Omega,\mathcal F,\mathbb P)$, and so the correlations decay fast enough to get the linear growth of the variance of $I_1$ after normalizing by $n^{2/r+1/p}$. We conclude that the contribution of this term to the $L^2(\mu_\omega)$ norm is just $O(n^{1/2+2/r+1/p})$.
Next, by Lemma \ref{GM} (applied for $s=2$),
$$
\max(\|I_2\|_{L^2(\mu_\omega)}, \|I_3\|_{L^2(\mu_\omega)})=O(n^{\frac{1}{p}+\frac1{r}+\delta}).
$$
Now, by \eqref{1226},
$$
\|I_4\|_{L^{\tilde s}(\mu_\omega)}^{\tilde s} =O(n^{ -\tilde s + 2 + 2\tilde s/p + 2\tilde s/r + 2\tilde s \delta}).
$$
Finally, \eqref{eq:U_nm} with $m=0$ gives that 
$$
\|I_5\|_{L^2(\mu_\omega)}^2=O(n^{1+\delta_*+ 2/p +4/r+ \delta}+n^{2+2/p+4/r+\delta_*(1-b)}).
$$
The proof of the proposition is complete, taking into account that $\tilde s<2$.
 \end{proof}

\section{Annealed limit theorems}
In this section, we will show how to obtain limit theorems for random variables of the form $(S_n \varphi)(\omega, \cdot)=\sum_{j=0}^{n-1}\varphi_{\sigma^j \omega}\circ T_\omega^j$ when viewed as random variables on the probability space $(\mathcal M, \mu)$, where $\mu$ is given by~\eqref{measmu}. We recall (see Proposition~\ref{properg}) that $\mu$ is ergodic for the skew product transformation $\tau$ (see~\eqref{spt}) provided that Assumption~\ref{P assum} holds. In order to control the size of the paper, we will focus only on the annealed version of Theorem~\ref{T-ASIP}.

We impose certain mixing assumptions on the base space $(\Omega, \mathcal F, \mathbb P, \sigma)$. 
More precisely, we assume that $(\Omega,\mathcal F,\mathbb P, \sigma)$ is the left-shift system generated by an $\alpha$-mixing stationary sequence $(X_j)_{j\in \mathbb Z}$. That is, if $(\mathcal Y,\mathcal G)$ is the common state space of $X_j$ then $\Omega=\mathcal Y^\mathbb Z, \mathcal F=\mathcal G^\mathbb Z$, $\mathbb P$ is the law induced by $(X_j)_{j\in\mathbb Z}$ on $\Omega$ and $\sigma:\Omega\to\Omega$ is the left-shift. We recall that the $\alpha$-mixing coefficients of $(X_j)$ are defined by 
$$
\alpha(n)=\sup_{k\in\mathbb Z}\left\{\left|\mathbb P(A\cap B)-\mathbb P(A)\mathbb P(B)\right|: A\in\mathcal F_{-\infty,k}, B\in\mathcal F_{k+n,\infty}\right\}
$$
where $\mathcal F_{a,b}$ is the $\sigma$-algebra generated by $X_s$ for all finite $a\leq s\leq b$.

In addition, throughout this section, we assume $M_\omega=M$ and that the measures $\mu_\omega$ are of the form $d\mu_\omega=h_\omega \, dm$, where $m$ is a  Borel probability measure in $M$ and $h_\omega$ are densities with respect to $m$. Note that in this case $\mathcal M=\Omega \times M$. We assume that $T_\omega^*m <<m$ for $\omega \in \Omega$, and let $\mathcal L_\omega$ be the corresponding transfer operator associated with $T_\omega$ and the measure $m$. Let $\mathcal L_\omega^n$ be defined as in~\eqref{Ln} replacing $L_\omega$ with $\mathcal L_\omega$.

By $\mathcal K$ we will denote the transfer operator of $\tau$ with respect to the measure $\mu$ and the sub-$\sigma$-algebra $\mathcal F_0$ of $\Omega \times M$ generated by the projection 
$$
\pi(\omega,x)=((\omega_j)_{j\geq 0},x), \,\,\omega=(\omega_j)_{j\in\mathbb Z}.
$$
\begin{proposition}\label{MainL}
 Assume that the following holds:
 \begin{itemize}
\item Assumption~\ref{P assum} holds  with  $\mathcal I=\{s\}$, $p=s$ and
$a_n=o(n^{-t})$ for $t>\frac 1 p$. In addition, 
\begin{equation}\label{Exp2}
\|\mathcal L_\omega^ng-m_\omega (g)h_{\sigma^n\omega}\|_{L^{p}(m)}\le CK(\omega)a_n (1+\|g\|_{\Lip}),
\end{equation}
for $\mathbb P$-a.e. $\omega \in \Omega$, $n\in \mathbb N$ and $g\colon M \to \mathbb R$ Lipschitz;
\item $\varphi \colon \Omega \times M \to \mathbb R$ is measurable, $\varphi_\omega:=\varphi (\omega, \cdot)$ is Lipschitz,
\[
\omega \mapsto \|\varphi_\omega \|_{\Lip}\in L^p(\Omega, \mathcal F, \mathbb P) \quad \text{and} \quad \int_{\Omega \times M}\varphi\, d\mu=0;
\]
\item   there is $c>0$ such that 
\begin{equation}\label{ulb}
h_\omega \ge c, \quad \text{for $\mathbb P$-a.e. $\omega \in \Omega$;}
\end{equation}
\item  $\omega\to \varphi(\omega,\cdot)$ and $\omega\to \mathcal L_\omega$ depend only on the coordinate $\omega_0$.
 \end{itemize}
 Then
\begin{equation}\label{1st}
\left\|\mathcal K^n \varphi\right\|_{L^{p/2}(\mu)}\leq C\big(a_{[n/2]}+(\alpha(n/2))^{1/p}\big)=:\gamma_n,
\end{equation}
if $p\ge 2$, where $C=C_{\varphi}>0$ is a constant. Moreover,
\begin{equation}\label{2nd}
\left\|\mathcal K^i(\varphi \mathcal K^j \varphi)-\mu\big(\varphi \mathcal K^j\varphi \big)\right\|_{L^{p/3}(\mu)}\leq  C\gamma_{\max(i,j)},
\end{equation}
provided that $p\ge 3$.
\end{proposition}

\begin{proof}
Let us first prove \eqref{1st}.
Let $q$ be the conjugate exponent of $p/2$.
Since $L^{q}(\mu)$ is the dual of $L^{p/2}(\mu)$  and $G$ and $\mathcal K^n\varphi$ are $\cF_0$-measurable, it is enough to show that for every $g\in L^q(\mathcal M,\mathcal F_0,\mu)$ with $\|g\|_{L^q(\mu)}\leq 1$ we have 
$$
\left|\int_{\Omega \times M} g\cdot(\mathcal K^n \varphi)d\mu\right|\leq\gamma_n.
$$
To achieve this, since $\mathcal K^n$ is the dual of the restriction of the Koopman operator $f\to f\circ\tau^n$ acting on $\cF_0$-measurable functions,
\begin{equation}\label{CrucialRel}
\begin{split}
\int_{\Omega \times M} g\cdot(\mathcal K^n \varphi)\, d\mu=\int_{\Omega \times M} \varphi\cdot (g\circ\tau^n) \, d\mu &=
\int_{\Omega }\left(\int_{M} \varphi_\omega \cdot (g_{\sigma^n\omega }\circ T_\omega^n)\, d\mu_\omega\right)\, d\mathbb P(\omega)\\
&=\int_\Omega\left(\int_{M} (L_\omega^n \varphi_{\omega})\cdot g_{\sigma^n\omega}\, d\mu_{\sigma^n\omega}\right)\, d\mathbb P(\omega),
\end{split}
\end{equation}
where $g_\omega:=g(\omega, \cdot)$.

By \eqref{dec2},
$$
\left\|L_\omega^n \varphi_{\omega}-\mu_\omega(\varphi_\omega)\right\|_{L^{p/2}(\mu_{\sigma^n\omega})}\le 
\left\|L_\omega^n \varphi_{\omega}-\mu_\omega(\varphi_\omega)\right\|_{L^p(\mu_{\sigma^n\omega})}\leq K(\omega) (1+\|\varphi_\omega \|_{\Lip})a_n,
$$
for $\mathbb P$-a.e. $\omega \in \Omega$ and $n\in \mathbb N$.
Hence, also using the $\sigma$-invariance of $\mathbb P$ and the H\"older inequality, 
$$
\int_{\Omega \times M} g\cdot(\mathcal K^n \varphi)\, d\mu=\int_\Omega \mu_\omega(\varphi_{\omega})\mu_{\sigma^n\omega}(g_{\sigma^n\omega})\, d\mathbb P(\omega)+I,
$$
where
\[
I:=\int_\Omega \left (\int_M (L_\omega^n \varphi_\omega-\mu_\omega(\varphi_\omega))g_{\sigma^n \omega}\, d\mu_{\sigma^n \omega}\right )\, d\mathbb P(\omega)
\]
and $|I|\leq C a_n$ for some constant $C>0$ independent of $n$.  In fact, 
\[
C=\|K\|_{L^p(\mathbb P)}\cdot \|g\|_{L^q(\mu)}\cdot \left (1+ \| \|\varphi_\omega\|_{\Lip}\|_{L^p(\mathbb P)}\right ).
\]

Next, note that
$$
\mu_{\sigma^n\omega}(g_{\sigma^n\omega})=m(g_{\sigma^n\omega}h_{\sigma^n\omega}).
$$
By \eqref{Exp2} we have 
$$
\left\|h_{\sigma^n\omega}-\mathcal L_{\sigma^{[n/2]}\omega}^{n-[n/2]}\textbf{1}\right\|_{L^p(m)}\leq CK(\sigma^{[n/2]}\omega)a_{[n/2]}
$$
for $\mathbb P$-a.e. $\omega \in \Omega$ and $n\in \mathbb N$, where $C>0$ is independent of these variables.  Therefore, by the H\"older inequality, 
\begin{equation}\label{G est}
\begin{split}
\left|m(g_{\sigma^n\omega}h_{\sigma^n\omega})-m(g_{\sigma^n\omega}\cL_{\sigma^{[n/2]}\omega}^{n-[n/2]}\textbf{1})\right| &\leq CK(\sigma^{[n/2]}\omega)a_{[n/2]}\|g_{\sigma^n\omega}\|_{L^q(m)}\\
&\leq Cc^{-1}K(\sigma^{[n/2]}\omega)a_{[n/2]}\|g_{\sigma^n\omega}\|_{L^q(\mu_{\sigma^n\omega})},
\end{split}
\end{equation}
for $\mathbb P$-a.e. $\omega \in \Omega$ and $n\in \mathbb N$, where $C>0$ is independent of these. We note that in the last inequality above, we used~\eqref{ulb}. Hence, by the H\"older inequality, we have
\begin{equation}\label{IJ}
\int_{\Omega \times M} g\cdot(\mathcal K^n \varphi)\, d\mu=\int _\Omega \mu_\omega(\varphi_{\omega})m(g_{\sigma^n\omega}\cL_{\sigma^{[n/2]}\omega}^{n-[n/2]}\textbf{1})\, d\mathbb P(\omega)+I+J,
\end{equation}
where 
\[
J:=\int_\Omega \mu_\omega(\varphi_\omega)\left (m(g_{\sigma^n\omega}h_{\sigma^n\omega})-m(g_{\sigma^n\omega}\cL_{\sigma^{[n/2]}\omega}^{n-[n/2]}\textbf{1})\right )\, d\mathbb P(\omega),
\]
and $|J|\leq Ca_{[n/2]}$ for some $C>0$ independent of $n$.

Next, using \eqref{Exp2}, $a_n=O(n^{-t})$ for $t>\frac 1 p$, $K\in L^p(\Omega, \mathcal F, \mathbb P)$ and Lemma~\ref{CMPlemma}, we have that \[h_\omega=\lim_{n\to\infty}\cL_{\sigma^{-n}\omega}^n\textbf{1} \quad \text{in $L^p(m)$, for $\mathbb P$-a-e. $\omega \in \Omega$.}\] Therefore, for $\mathbb P$-a.e. $\omega \in \Omega$,  $h_\omega$ depends only on the coordinates $\omega_j$ for $j\leq 0$ and consequently 
$$
\mu_\omega(\varphi_{\omega})=F(\omega_j; j\leq 0),
$$
for some measurable function $F$ such that $\|F\|_{L^p(\mathbb P)}\leq \|\varphi\|_{L^p(\mu)}$. Observe also that the random variable 
$$
A_n(\omega)=m(g_{\sigma^n\omega}\mathcal L_{\sigma^{[n/2]}\omega}^{n-[n/2]}\textbf{1})
$$
depends only on $\omega_j$, $j\geq [n/2]$ since $g_\omega(x)$ is a function of $x$ and $\omega_j, j\geq0$ (i.e. it factors through $\pi$).
Due to~\eqref{ulb} we have
$$
|A_n(\omega)|=\left|\mu_{\sigma^n\omega}\big(g_{\sigma^n\omega}L^{n-[n/2]}_{\sigma^{[n/2]}\omega}(1/h_{\sigma^{[n/2]}\omega})\big)\right|\leq \frac 1 c \mu_{\sigma^n\omega}(|g_{\sigma^n\omega}|),
$$
for $\mathbb P$-a.e. $\omega \in \Omega$ and $n\in \mathbb N$.
Thus, using also \cite[Eq. (1.2.17)]{HK}  we see that there is a constant $C>0$ such that
$$
\left|\int_\Omega \mu_\omega(\varphi_{\omega})m(g_{\sigma^n\omega}\cL_{\sigma^{[n/2]}\omega}^{n-[n/2]}\textbf{1})\, d\mathbb P(\omega)\right|\leq C(\alpha(n/2))^{1-1/p-1/q}
$$
for $n\in \mathbb N$,
where we have taken into account that $\int_\Omega \mu_\omega(\varphi_\omega)\, d\mathbb P(\omega)=\int_{\Omega \times M}\varphi \, d\mu=0$. This, together with \eqref{IJ} and the previous estimates on $I$ and $J$, proves \eqref{1st}.

Now, let us prove \eqref{2nd}.  First, since $\mathcal K$ weakly contracts the $L^s(\mu)$ norms (being defined through conditional expectations) we have
$$
\left\|\mathcal K^i(\varphi\mathcal K^j \varphi)-\mu\big(\varphi \mathcal K^j \varphi\big)\right\|_{L^{p/3}(\mu)}\leq 2\|\varphi\|_{L^{p}(\mu)} \cdot \|\mathcal K^j\varphi\|_{L^{p/2}(\mu)}.
$$
This together with \eqref{1st} provides the desired estimate when $j\geq i$. The estimate in case $i>j$ is carried out similarly to the proof of \eqref{1st}. Let $u$ be the conjugate exponent of $p/3$ and let $g\in L^u(\Omega\times M, \mathcal F_0, \mu)$ be such that $\|g\|_{L^u(\mu)}\leq 1$. Let us first show that 
\begin{equation}\label{First}
\int_{\Omega \times M} \mathcal K^i(\varphi\mathcal K^j\varphi)g\,d\mu=\int_\Omega \mu_{\omega}\big(\varphi_\omega\cdot (\varphi_{\sigma^j\omega}\circ T_\omega^j)\big)\mu_{\sigma^{i+j}\omega}(g_{\sigma^{i+j}\omega})\, d\mathbb P(\omega)+I,
\end{equation}
where $|I|\leq C_2\gamma_{i}$ and $C_2>0$ is some constant.

In order to prove \eqref{First}, using that $\mathcal K$ satisfies the duality relation and the disintegration $\mu=\int \mu_\omega d\mathbb P(\omega)$, we first have
\begin{equation}\label{Eq}
\begin{split}
\int_{\Omega \times M} \mathcal K^i(\varphi\mathcal K^j \varphi)g\,d\mu &=\int_{\Omega \times M} (\varphi \mathcal K^j\varphi)\cdot g\circ\tau^i\,d\mu \\
&=
\int_{\Omega \times M} \mathcal K^j \varphi \cdot\big(\varphi \cdot (g\circ\tau^i)\big)d\mu \\ &=\int_{\Omega \times M}
\big(\varphi\cdot(\varphi \circ\tau^j)\big)\cdot g\circ\tau^{i+j}\,d\mu\\
&=\int_\Omega \left (\int_M \varphi_\omega\cdot (\varphi_{\sigma^j\omega}\circ T_\omega^{j})\cdot (g_{\sigma^{i+j}\omega}\circ T_\omega^{i+j})d\mu_\omega\right)d\mathbb P(\omega)\\
&=\int_\Omega \left(\int_M L^{i+j}_\omega\big(\varphi_\omega \cdot (\varphi_{\sigma^j\omega}\circ T_\omega^j)\big)g_{\sigma^{i+j}\omega}d\mu_{\sigma^{i+j}\omega}\right)d\mathbb P(\omega).
\end{split}
\end{equation}
Next, since $L_\omega^n((f\circ T_\omega^n)g)=fL_\omega^n g$ for every functions $f, g$ and $n\in \mathbb N$, we have
$$
L^{i+j}_\omega\big(\varphi_\omega\cdot (\varphi_{\sigma^j\omega}\circ T_\omega^j)\big)=L_{\sigma^j\omega}^i(\varphi_{\sigma^j\omega}L_{\omega}^j\varphi_\omega).
$$
Thus,
$$
\int_{\Omega \times M} \mathcal K^i(\varphi \mathcal K^j\varphi)g\,d\mu=\int_\Omega \left(\int_ML_{\sigma^j\omega}^i(\varphi_{\sigma^j\omega}L_{\omega}^j\varphi_\omega)g_{\sigma^{i+j}\omega}d\mu_{\sigma^{i+j}\omega}\right)d\mathbb P(\omega).
$$
Now \eqref{First} follows from centering the above integrand and then using \eqref{dec2} and the H\"older inequality. 

It remains to estimate 
$$
\int_\Omega \mu_{\omega}\big(\varphi_\omega\cdot (\varphi_{\sigma^j\omega}\circ T_\omega^j)\big)\mu_{\sigma^{i+j}\omega}(g_{\sigma^{i+j}\omega})\, d\mathbb P(\omega).
$$
This is done exactly as in the proof of \eqref{1st}. In fact, $g_{\sigma^{i+j}\omega}$ depends only on $\omega_k$, $k\geq i+j$ and $\mu_{\sigma^{i+j}\omega}$ can be approximated on average by $\cF_{i+j-m,i+j+m}$ measurable functions within $a_{m}$ (that is, $a_m$ controls the error term). Similarly $\varphi_\omega\cdot (\varphi_{\sigma^j\omega}\circ T_\omega^j)$ can be approximated by $\cF_{-\infty,j+m}$ measurable functions within $O(a_m)$. Taking $m=[i/2]$ yields \eqref{2nd}. 
\end{proof}

We are now in a position to formulate an annealed version of Theorem~\ref{T-ASIP}.
\begin{theorem}
Let the assumptions of Proposition~\ref{MainL} be in force with $p=8$ and $t>16$. In addition, suppose that $\alpha(n)=O(n^{-t})$. Then the quantity
\[
\Sigma^2=\int_{\Omega \times M}\varphi^2\, d\mu+2\sum_{n=1}^\infty\int_{\Omega \times M}\varphi \cdot (\varphi\circ \tau^n)\, d\mu
\]
is finite and nonnegative. Moreover, provided that $\Sigma^2>0$,  by enlarging the probability space $(\Omega \times M, \mu)$ if necessary, there is a sequence $(Z_i)_i$ of i.i.d Gaussian random variables with mean $0$ and variance $\Sigma^2$ such that 
\[
\sup_{1\le k\le n} \left |\sum_{i=0}^{k-1}\left (\varphi \circ \tau^i-\int_{\Omega \times M}\varphi\, d\mu\right )-Z_i\right |=O(n^{1/4}(\log n)^{1/2}(\log \log n)^{1/4}), \quad \text{$\mu$-a.s.}
\]\end{theorem}

\begin{proof}
Throughout the proof, $C$ will denote a generic positive constant independent on $n$.
It follows from~\eqref{1st} that 
\[
\begin{split}
\sum_{n\ge 2}n^{5/2}(\log n)^3\|\mathcal K^n \varphi\|_{L^4(\mu)}^4 &\le C\sum_{n\ge 2}n^{5/2}(\log n)^3\big(a_{[n/2]}+(\alpha(n/2))^{1/8}\big)^4 \\
&\le C\sum_{n\ge 2}n^{5/2}(\log n)^3 \left (a_{[n/2]}^4+(\alpha(n/2))^{1/2}\right)\\
&\le C\sum_{n\ge 2}n^{5/2}(\log n)^3 n^{-t/2}<+\infty,
\end{split}
\]
as $t>7$. Similarly, applying~\eqref{1st} with $p=4$ we have 
\[
\begin{split}
\sum_{n\ge 2}n(\log n)^3\|\mathcal K^n \varphi\|_{L^2(\mu)}^2 &\le C \sum_{n\ge 2}n(\log n)^3 n^{-t/2}<+\infty,
\end{split}
\]
as $t>4$. Hence, \cite[(3.6)]{CM} holds.

On the other hand, \eqref{2nd} for $p=6$ gives 
\[
\begin{split}
\sum_{n\ge 2}\frac{(\log n)^3}{n^2}\left (\sum_{i=1}^n\sum_{j=0}^{n-i}\left\|\mathcal K^i(\varphi \mathcal K^j \varphi)-\mu\big(\varphi \mathcal K^j\varphi \big)\right\|_{L^{2}(\mu)}\right )^2 &\le 
\sum_{n\ge 2}\frac{(\log n)^3}{n^2}\left (\sum_{i=1}^n\sum_{j=0}^{n-i}\gamma_{\max \{i, j\}}\right)^2,
\end{split}
\]
with $\gamma_n=O(n^{-t/8})$. It is easy to show that
\[
\sum_{i=1}^n\sum_{j=0}^{n-i}\gamma_{\max \{i, j\}}=O(1),
\]
since $t>16$. Consequently, 
\[
\sum_{n\ge 2}\frac{(\log n)^3}{n^2}\left (\sum_{i=1}^n\sum_{j=0}^{n-i}\left\|\mathcal K^i(\varphi \mathcal K^j \varphi)-\mu\big(\varphi \mathcal K^j\varphi \big)\right\|_{L^{2}(\mu)}\right )^2<+\infty,
\]
which yields that~\cite[(3.7)]{CM} holds. The conclusion of the theorem now follows from~\cite[Theorem 3.2]{CM}.
\end{proof}
\begin{remark}
Similarly to the previous section, we can also show that $\|S_n\varphi\|_{L^s}=O(\sqrt n)$ for $s
\leq p/2$ and get CLT rates $O(n^{-1/5})$. Indeed, these results relied only on martingale approximation. However, we decided not to include full statements in order not to overload the paper. 
\end{remark}
	
	\section{Revisiting quenched memory loss for random LSV maps}
	
	Let $(\Omega, \mathcal F, \mathbb P)$ be an arbitrary probability space equipped with an invertible $\mathbb P$-preserving transformation $\sigma \colon \Omega \to \Omega$. Throughout 
    this section we assume that $(\mathbb P,\sigma)$ is ergodic.
    By $m$, we denote the Lebesgue measure on $M:=[0, 1]$.
	Recall that for each $\beta\in (0, 1)$ the associated Liverani--Saussol--Vaienti (LSV) map (introduced in~\cite{LSV}) is given by 
	\[
	T_\beta(x)=\begin{cases}
		x(1+2^\beta x^\beta) & 0\le x<\frac 1 2 \\
		2x-1 & \frac 1 2 \le x\le 1.
	\end{cases}
	\]
	Let $\beta : \Omega \to (0,1)$ be a measurable map and, for 
	each each $\omega \in \Omega$, let $T_\omega  \colon M \to M$ be the LSV map with parameter $\beta(\omega) \in (0, 1)$.  Set 
    \begin{align}\label{eq:random_lsv}
	T_\omega^n:=T_{\sigma^{n-1}\omega}\circ \ldots \circ T_{\sigma \omega}\circ T_\omega, \quad \omega \in \Omega,  \ n\in \mathbb N.
	\end{align}
    Let $\mathcal L_\omega \colon L^1(m)\to L^1(m)$ denote the transfer operator associated 
    to $T_\omega$ and $m$, namely
    \begin{align*}
        \mathcal L_\omega \varphi (x) = \sum_{ y \in T_{\omega}^{-1}(\{ x \})  } 
        \frac{\varphi(y)}{ T_\omega'(y) }.
    \end{align*}
    We define 
	\[
	\mathcal L_\omega^n:=\mathcal L_{\sigma^{n-1}\omega}\circ \ldots \circ \mathcal L_{\sigma \omega}\circ \mathcal L_\omega, \quad \omega \in \Omega,  \ n\in \mathbb N.
	\]
	We also set $\mathcal L_\omega^0 :=\Id$ for $\omega \in \Omega$.
	In the sequel, we assume that
	\begin{align}\label{eq:ass_1}
		\beta:=\esssup_{\omega \in \Omega}\beta(\omega)< 1.
	\end{align}
    Next, let $\beta_0=\essinf_{\omega\in\Omega} \beta(\omega)$. Then for every $\beta_0<\gamma<1$ we have 
    \begin{align}\label{eq:ass_2}
		b_0 := \bP( \beta(\omega) \le \gamma  ) > 0.
	\end{align}
    Henceforth, we shall take an arbitrary $\gamma>0$ that satisfies \eqref{eq:ass_2}. When $\beta_0>0$ we can take $\gamma$ arbitrarily close to $\beta_0$, which would yield the best rates in what follows.
    \begin{remark}
    In many natural circumstances $\beta_0=0$. For example, this is the case when $\beta(\omega)$ is supported on some interval $(0,a), a<1$ and $\mathbb P(\beta(\omega)\in A)=\int_{A}f(x)dx$ for some positive density $f$. This is also trivially the case when $\bP(\beta(\omega)=0)>0$ (assuming that we allow zero values), and many other examples can be given. 

It should be noted that the case where $\beta$ takes the value $0$ on a set of positive probability fits the framework of the work of the second author~\cite{YH cones 2}. Under appropriate mixing assumptions for the driving system $(\Omega, \mathcal F, \mathbb P, \sigma)$ (similar to those that we will discuss in the sequel), one can establish a variety of limit theorems (see~\cite{YH cones 2}).  Therefore, in the present paper, we restrict our attention to the case where $\beta$ takes values in $(0, 1)$. 

\end{remark}

	It is proved in~\cite[Proposition 9]{DGTS} that the cocycle of maps $(T_\omega)_{\omega \in \Omega}$ admits a unique random a.c.i.m $\mu$ on $\Omega \times M$, which can be identified with a family of probability measures $(\mu_\omega)_{\omega \in \Omega}$ on $M$ such that
	\[
	T_\omega^*\mu_\omega=\mu_{\sigma \omega} \quad \text{for $\mathbb P$-a.e. $\omega \in \Omega$.}
	\]
	Moreover, $d\mu_\omega=h_\omega \, dm$ with $h_\omega \in \mathcal C_*\cap \mathcal C_2$ for some $a, b_1, b_2 > 0$, where $\mathcal C_*=\mathcal C_*(a)$ and $\mathcal C_2=\mathcal C_2(b_1, b_2)$ are cones as in~\cite[Section 2.2]{DGTS}. That is, $\mathcal C_*$ consists of  $\phi \in C^0(0, 1]\cap L^1(m)$ such that $\phi \ge 0$, $\phi$ is decreasing, $X^{\beta+1}\phi$ is increasing (where $X$ denotes the identity map), and 
	\[
	\int_0^x \phi (t)\, dt\le ax^{1-\beta}\int_0^1 \phi (t)\, dt \quad x\in (0, 1].
	\]
	Moreover, $\mathcal C_2$ consists of all $\phi \in C^2(0, 1]$ so that
	\[
	\phi (x)\ge 0, \quad |\phi'(x)| \le \frac{b_1}{x}\phi(x) \quad \text{and} \quad |\phi''(x)| \le \frac{b_2}{x^2}\phi(x), \quad x\in (0, 1].
	\]
	We stress that the parameters $a$, $b_1$, and $b_2$ depend only on $\beta$. Finally (see~\cite[Eq. (10)]{DGTS}), for $\mathbb P$-a.e. $\omega \in \Omega$,
    \[
    h_\omega=\lim_{n\to \infty}\mathcal L_{\sigma^{-n}\omega}^n 1 \quad \text{in $L^1(m)$.}
    \]

	For $\delta \in \{ \sigma, \sigma^{-1} \}$, define 
	$S_n^{\delta}(\omega) = \sum_{i=0}^{n-1} \psi \circ  \delta^i(\omega)$, where  $\psi(\omega) = \mathbf{1}_{  (0, \gamma] } ( \beta(\omega) )$. Note that
	$$
	S_n^\sigma(\omega) = \# \{  0 \le j \le n - 1 \: : \:  \beta( \sigma^j \omega ) \le \gamma  \}.
	$$
	Given $\ve \in (0,1/2)$, define 
	\begin{align}\label{eq:N_eps}
		N_\ve(\omega) = \max_{ \delta \in \{\sigma, \sigma^{-1} \} } \max \biggl( \biggl\{ n \ge 1 \: : \:  n^{-1} S_n^{\delta }(\omega) \notin [ b_0(1-\ve), b_0(1+\ve) ]  \biggr\} \cup \{0\} \biggr).
	\end{align}
	From Birkhoff's ergodic theorem it follows that $N_\ve(\omega) < \infty$ for $\bP$-a.e. $\omega \in \Omega$. By definition,
	\begin{align}\label{eq:large_n}
		n \ge N_\ve(\omega) \implies  | n^{-1}S_n^\delta ( \omega ) - b_0  | \le \ve b_0, \quad \delta \in \{\sigma, \sigma^{-1}\}.
	\end{align}
		Let $L_\omega \colon L^1(\mu_\omega)\to L^1(\mu_{\sigma \omega})$ be given by 
	\begin{equation}\label{Lom}
	L_\omega \varphi=\frac{\mathcal L_\omega(\varphi h_\omega)}{h_{\sigma \omega}}, \quad \varphi \in L^1( \mu_\omega).
	\end{equation}
    We recall that, since $h_\omega \in \cC_*(a)$, there exists $b > 0$ depending only on 
    $\beta$ and $a$ such that $\inf_{x \in M} h_\omega(x) \ge b$ for $\mathbb{P}$-a.e. 
    $\omega \in \Omega$; see \cite[Lemma 2.4]{LSV}.
	
	\begin{theorem}\label{thm:quenched_ml} Assume \eqref{eq:ass_1} and \eqref{eq:ass_2}. 
		Let $g_i : M \to \bR$ be Lipschitz continuous, $i = 1,2$.
		There exists $\tilde{\ve} \in (0, 1/2)$ depending only on the random dynamical system such that for any $\ve \in (0, \tilde{\ve}]$, any $0 \le s \le i < j$, and for $\bP$-a.e. $\omega \in \Omega$, 
		\begin{align}\label{eq:quenched_ml_1}
			\begin{split}
				& \bigl\Vert  \bigl[  L_{  \sigma^i \omega }^{j-i} \bigl(g_2  L_{ \sigma^s \omega }^{i-s}   (g_1 \bigr)  \bigr]_{ \sigma^j \omega }  \bigr\Vert_{  L^1(  \mu_{ \sigma^j \omega } ) }
				\\
				&\le C    (1 +  \Vert  g_1 \Vert_{ \Lip } )  (1 +  \Vert g_2 \Vert_{ \Lip } )
				(
				1 + N_\ve( \sigma^i \omega) )^{ 1 \vee (1/ \gamma  -  1 ) }  (j - i)^{-1/\gamma + 1}.
			\end{split}
		\end{align}
		Here, 
		$C$ is a positive constant depending only on the random dynamical system (RDS). 
	\end{theorem}

	The proof of the theorem is given in the Appendix A. 
	As a straightforward consequence, we obtain the following 
	memory loss estimates in $L^p$.
	
	\begin{corollary}\label{cor:lp_ml}
	Let $p \ge 1$. 
	In the setting of Theorem \ref{thm:quenched_ml}, we have the following estimate for any $\ve \in (0, \tilde{\ve}]$, any $0 \le s  \le i < j$, and for $\bP$-a.e. $\omega \in \Omega$: 
	\begin{align}\label{eq:quenched_ml_lp}
		\begin{split}
			& \bigl\Vert  \bigl[  L_{  \sigma^i \omega }^{j-i} \bigl( g_2  L_{ \sigma^s \omega }^{i-s}   ( g_1 )  \bigr)  \bigr]_{ \sigma^j \omega }  \bigr\Vert_{  L^p(  \mu_{ \sigma^j \omega } ) }
			\\
			&\le C_p    (1 +  \Vert  g_1 \Vert_{ \Lip } )  (1 +  \Vert  g_2 \Vert_{ \Lip } )
			(1 + N_\ve( \sigma^i \omega) )^{ \frac1p(  1 \vee (1/ \gamma  -  1 ) ) }  (j - i)^{ - \frac1p( \frac{1}{\gamma} - 1)},
		\end{split}
	\end{align}
	where $C_p>0$ is a constant depending only on the RDS and $p$.
	\end{corollary}
	
	\begin{proof}
	The proof is similar to the proof of \cite[Proposition 3.5]{NPT}. Namely, we use
	\begin{align*}
		\int_0^1 |f(x) |^p \, d \mu_{ \sigma^i \omega }(x)
		\le \Vert f  \Vert_{L^\infty}^{p-1}
		\int_0^1 |f(x)| \, d \mu_{ \sigma^i \omega }(x),
	\end{align*}
	and the observation that, for $\bP$-a.e. $\omega \in \Omega$,
	\begin{align*}
	\bigl|L_{  \sigma^i \omega }^{j-i} \bigl( g_2 L_{ \sigma^s \omega }^{i-s}   (  g_1 ) 
	 \bigr) \bigr| 
	&= \bigl| 
	h_{ \sigma^j \omega }^{-1} \cL_{ \sigma^i \omega }^{j-i} \bigl(  g_2 
	h_{\sigma^i \omega}  h_{\sigma^i \omega}^{-1} \cL_{\sigma^s \omega}^{i-s}
	( h_{\sigma^s \omega} g_1 )  
	\bigr)
	\bigr| 
	\le \Vert g_1 \Vert_{\Lip} \Vert g_2 \Vert_{\Lip}.
	\end{align*}
	It follows by \eqref{eq:quenched_ml_1} that 
	\begin{align*}
	&\bigl\Vert  \bigl[  L_{  \sigma^i \omega }^{j-i} \bigl( g_2  L_{ \sigma^s \omega }^{i-s}   ( g_1 )  \bigr)  \bigr]_{ \sigma^j \omega }  \bigr\Vert_{  L^p(  \mu_{ \sigma^j \omega } ) } \\
	&\le ( 2  \Vert g_1 \Vert_{\Lip} \Vert g_2 \Vert_{\Lip} )^{  \frac{p-1}{p} }
	\bigl\Vert 
	\bigl[  L_{  \sigma^i \omega }^{j-i} \bigl( g_2  L_{ \sigma^s \omega }^{i-s}   ( g_1 )  \bigr)  \bigr]_{ \sigma^j \omega } 
	  \bigr\Vert_{  L^1(  \mu_{ \sigma^j \omega } ) }^{ \frac1p } \\
	&\le C_p   (1 +  \Vert  g_1 \Vert_{ \Lip } )
	 (1 +  \Vert  g_2 \Vert_{ \Lip } )
	(1 + N_\ve( \sigma^i \omega) )^{ \frac1p(  1 \vee (1/ \gamma  -  1 ) ) }  (j - i)^{ - \frac1p( \frac{1}{\gamma} - 1)},
	\end{align*}
	as wanted.
	\end{proof}

	\begin{remark}
	Theorem \ref{thm:quenched_ml} is proved by adapting the proof of \cite[Theorem 2.6]{KL2}. 
	The method used there also applies to the following family of maps (see \cite[Section 3.2]{KL2}), introduced by Pikovsky \cite{P} and studied in \cite{MR, CHMV}: 
	for $\alpha > 1$, define $T_\alpha$ 
    on $[0,1]$
    implicitly by
	\begin{align}\label{eq:pikovsky}
		x = \begin{cases}
			\frac{1}{2 \alpha } ( 1 + T_\alpha(x))^{\alpha},  &0 \le x \le \frac{1}{2 \alpha }, \\
			T_\alpha(x) + \frac{1}{2 \alpha } (1 - T_\alpha(x))^{\alpha}, &\frac{1}{2 \alpha } \le x \le 1,
		\end{cases}
	\end{align}
	and extend to a map $T_\alpha : [-1,1] \to [-1,1]$ by setting $T_\alpha(x) = - T_\alpha(-x)$ for $x \in [-1,0]$. 
	This map has neutral fixed points
	at $x = 1,-1$, while at $x = 0$ its derivative becomes infinite.
    The graph of $T_\alpha$ is illustrated in Figure
    \ref{fig:pikovsky}.
	For each $\alpha > 1$,
	$T_\alpha$ preserves the Lebesgue 
	measure $\hat{m}$ on $[-1,1]$ normalized to probability.
	Consider random compositions of Pikovsky maps 
	$T_{ \beta( \sigma^{n-1}\omega ) }\circ \ldots \circ T_{ \beta( \sigma \omega ) }\circ T_{ \beta(\omega) }$ 
	with an
	ergodic driving system $\sigma$
	as in 
	\eqref{eq:random_lsv}, and assume that 
	$\beta : \Omega \to (1, \infty)$ satisfies
	the following conditions with 
	$1 < \gamma_{-} \le \gamma < 2 \le \gamma_{+} < 3$:
	\begin{itemize}
		\item $\gamma_{-} < \textnormal{essinf}_{\omega \in \Omega}  \beta(\omega) \le \esssup_{\omega \in \Omega} \beta(\omega) \le \gamma_+$, and 
		\item $\bP( \beta(\omega) \le \gamma ) > 0$.
	\end{itemize}
	Define $\mu_\omega = \hat{m}$, which trivially
	satisfies
	$(T_\omega)_*\mu_\omega = \mu_{\sigma \omega}$.
	Then,
	the following quenched memory loss estimate
	for Lipschitz functions $g_1,g_2: [-1,1] \to \bR$
	can be obtained by modifying the 
	proof of 
	\eqref{eq:quenched_ml_1}, 
	using
	results from \cite[Section 3.2]{KL2}:
	for any
	$\ve \in (0, \tilde{\ve}]$
    with $\tilde\ve$ sufficiently small, $0 \le s  \le i < j$, and 
	for $\bP$-a.e. $\omega \in \Omega$,
	\begin{align*}
			& \bigl\Vert  \bigl[  L_{  \sigma^i \omega }^{j-i} \bigl( g_2  L_{ \sigma^s \omega }^{i-s}   ( g_1 )  \bigr)  \bigr]_{ \sigma^j \omega }  \bigr\Vert_{  L^1(   \hat{m} ) }
			\\
			&\le C   (1 +  \Vert  g_1 \Vert_{ \Lip } )  (1 +  \Vert  g_2 \Vert_{ \Lip } )
			(1 + N_\ve( \sigma^i \omega) )^{ 
			( 1 \vee  \frac{ 1  }{\gamma - 1} )
			}  (j - i)^{ - \frac{ 1  }{\gamma - 1}   }.
	\end{align*}
	Here, 
	$N_\ve( \omega)$ is defined as in
	\eqref{eq:N_eps}. A corresponding estimate in 
    $L^p$ follows by the same argument 
    as in the proof of \eqref{eq:quenched_ml_lp}.
	\end{remark}

    \begin{figure}[h!]
    \centering
    \begin{tikzpicture}[every node/.style={scale=0.7}]
    \tikzmath{
        \ggamma = 1.5;
    }
    \begin{axis}[
        height       = 2.5in,
        xmin         = -1.0,
        xmax         = 1.1,
        ymin         = -1.0,
        ymax         = 1.1,
        xtick        = {-1.0, 0.0, 1.0},
        xticklabels  = {$-1$, $0$, $1$},
        axis x line  = bottom,
        ytick        = {-1.0, 0.0, 1.0},
        yticklabels  = {$-1$, $0$, $1$},
        axis y line  = left,
        line cap=round
        ]
        \addplot[gray,dotted] coordinates { (-1,1) (1,1) };
        \addplot[gray,dotted] coordinates { (1,1) (1,-1) };
        \addplot[gray,dotted] coordinates { (-1,-1) (1,1) };
        \addplot[gray,dotted] coordinates { (0,-1) (0,1) };
        \addplot[red,thick,domain=-1.0:0.0, samples=21]({0.5 / \ggamma * (1 + x)^\ggamma}, {x});
        \addplot[red,thick,domain=0.0:1.0, samples=21]({x + 0.5 / \ggamma * (1 - x)^\ggamma}, {x});
        \addplot[red,thick,domain=-1.0:0.0, samples=21]({- 0.5 / \ggamma * (1 + x)^\ggamma}, {-x});
        \addplot[red,thick,domain=0.0:1.0, samples=21]({- x - 0.5 / \ggamma * (1 - x)^\ggamma}, {-x});
    \end{axis}
    \end{tikzpicture}
    \caption{Graph of the Pikovsky map \eqref{eq:pikovsky}.}
    \label{fig:pikovsky}
    \end{figure}
    
	\subsection{$\alpha$-mixing driving process}

	We apply Theorem \ref{thm:quenched_ml} in the case of an $\alpha$-mixing driving process. 
	Below, we consider the $\alpha$-mixing coefficients for the stationary process $(  \beta \circ \sigma^{i}  )_{i \in \mathbb{Z}}$,  defined by
	$$
	\alpha(n) = \sup_{i \in \mathbb{Z}}
    \{ | \bP(A \cap B) - \bP(A) \bP(B) | \: : \: 
    A \in \mathcal{F}_{-\infty, i}, \: 
    B \in \mathcal{F}_{i+n, \infty}
    \}.
	$$
	Here, $\cF_{-\infty, i}$ is the sub-sigma-algebra generated by $( \beta \circ \sigma^j )_{j \le i}$ and 
	$\cF_{i,\infty}$ is the sub-sigma-algebra generated by $( \beta \circ \sigma^j )_{j \ge i}$.
	
	We assume \eqref{eq:ass_1} and recall \eqref{eq:ass_2}.

	\begin{corollary}\label{cor}Let $\gamma < 1/2$ and $p,s \ge 1$.
		Suppose that
        \begin{align}\label{eq:mixing_rate}
			\alpha(n)  = O( n^{ -q + 1 } \log^{-\iota} (n)  )
		\end{align}
		holds with
		\begin{equation}\label{constants}
		\iota > q \quad \text{and} \quad q > \frac{p}{s} \biggl( \frac{1}{\gamma} - 1 \biggr) + 2.
		\end{equation}
		Then, there exists $B_s \in L^p(\Omega, \cF, \bP)$ such that for any $0 \le r  \le i < j$, and 
		for $\bP$-a.e. $\omega \in \Omega$, for every $g_1, g_2: M \to\mathbb R$ we have\begin{align}\label{eq:lp_ml_2}
			\begin{split}
				& \bigl\Vert  \bigl[  L_{  \sigma^i \omega }^{j-i} \bigl( g_2 L_{ \sigma^r \omega }^{i-r}   ( g_1)  \bigr)  \bigr]_{ \sigma^j \omega }  \bigr\Vert_{  L^s(  \mu_{ \sigma^j \omega } ) }
				\\
				&\le C    (1 +  \Vert g_1\Vert_{ \Lip } )  (1 +  \Vert  g_2 \Vert_{ \Lip } )
				B_s( \sigma^i \omega )  (j - i)^{ - \frac1s( \frac{1}{\gamma} - 1)},
			\end{split}
		\end{align}
            where $C$ is a constant depending only on $s,p,q$, and the RDS.
	\end{corollary}

	\begin{proof} By \eqref{eq:quenched_ml_lp}, there exists $\ve \in (0, 1/2)$ 
		and $C > 0$
		depending only on the RDS such that \eqref{eq:lp_ml_2} holds with
		$$
		B_s(\omega) := (1 + N_\ve(\omega))^{ \frac1s ( \frac1\gamma - 1 )  } \in L^p(\Omega, \cF, \bP).
		$$
		Therefore, it suffices to show that $B_s \in L^p(\Omega, \cF, \bP)$.

	By the definition of $N_\ve(\omega)$, we have
	$$
	\{ N_\ve \ge k \} \subset \bigcup_{ \delta \in \{ \sigma, \sigma^{-1} \} }   \{  \sup_{\ell \ge k} | \ell^{-1} \tilde{S}_{\ell}^{\delta}  |  > b_0 \ve  \},
	$$
	where
	$$
	\tilde{S}_{\ell}^{\delta} = S_{\ell}^{\delta} - \bE( S_{\ell}^{\delta} ) = S_{\ell}^{\delta} - \ell b_0.
	$$
	Consequently,
	\begin{align*}
			\int_{\Omega} B_s(\omega)^p \, d \bP(\omega)
			&\le C_{p, s, \gamma} \sum_{k=1}^\infty k^{  \frac{p}{s} (  \frac1\gamma - 1 ) - 1 } \bP( N_\ve \ge k ) \\
			&\le C_{p, s, \gamma} \sum_{ \delta \in \{ \sigma, \sigma^{-1} \} } \sum_{k=1}^\infty  k^{  \frac{p}{s} (  \frac1\gamma - 1 ) - 1 } \bP( \sup_{\ell \ge k} | \ell^{-1} \tilde{S}_{\ell}^{\delta}  |  > b_0 \ve ) \\
			&\le C_{p, s, \gamma} \sum_{ \delta \in \{ \sigma, \sigma^{-1} \} } \sum_{k=1}^\infty  k^{  \frac{p}{s} (  \frac1\gamma - 1 ) - 1 }
			\sum_{\ell \ge k}
			\bP( | \ell^{-1} \tilde{S}_{\ell}^{\delta}  |  \ge b_0 \ve ).
		\end{align*}
		Since $\Vert \beta\circ \sigma^{j} \Vert_\infty \le 1$, 
		the strong law of large numbers for $\alpha$-mixing sequences in
		\cite[Theorem 1]{S} applied with $r = \infty$ yields
		$$
		A_q := \sum_{\ell \ge 1} \ell^{ q - 2 } 
		\bP( | \ell^{-1} \tilde{S}_{\ell}^{\delta}  | \ge b_0 \ve  ) < \infty
		$$
		for $\delta \in \{ \sigma, \sigma^{-1} \}$, 
		assuming \eqref{eq:mixing_rate} with $\iota > q$. Therefore,
		\begin{align*}
			\int_{\Omega} B_s(\omega)^p \, d \bP(\omega) \le 
			2A_q C_{p, s, \gamma} \sum_{k=1}^\infty k^{  \frac{p}{s} (  \frac1\gamma - 1 ) - q + 1 } < \infty,
		\end{align*}
		due to the second requirement in~\eqref{constants}.
	\end{proof}

    \begin{remark}\label{remdriving}
The requirement~\eqref{eq:mixing_rate} is trivially satisfied in an i.i.d setting. More generally, it is satisfied when $(\Omega, \mathcal F, \mathbb P, \sigma)$ is a transitive subshift of finite type with a Gibbs measure (associated to some H\"older continuous potential) and $\beta(\omega)$ depends only on a finite number of coordinates of $\omega$.  In this case, $\alpha(n)=O(e^{-cn})$ for some $c>0$, and thus~\eqref{eq:mixing_rate} is satisfied with any $q, \iota>0$. Much broader classes of examples where~\eqref{eq:mixing_rate} holds can be obtained from the setup and discussion in~\cite[Section 2.1 and Example 2.6]{YH preprint}.

    \end{remark}

 We get the following consequence of Corollary \ref{cor}.
    \begin{corollary}\label{cltlsv}
    Let $\gamma  <\frac 12$ and $p\ge 1$.
    Suppose that~\eqref{eq:mixing_rate} holds with 
    \begin{equation}\label{854eq}
    \iota >q \quad \text{and} \quad q>p\left (\frac{1}{\gamma}-1\right )+2.
    \end{equation}
    Furthermore, 
    let $\varphi \colon \Omega \times M \to \mathbb R$ be a measurable map satisfying:
    \begin{itemize}
        \item $\int_M \varphi_\omega \, d\mu_\omega=0$ for $\mathbb P$-a.e. $\omega \in \Omega$, where $\varphi_\omega:=\varphi(\omega, \cdot)$;
        \item for  $\mathbb P$-a.e. $\omega \in \Omega$, $\varphi_\omega$ is Lipschitz and~\eqref{obsint} holds with $r>0$ such that $\frac 1 p +\frac 1 r \le \frac 12$.
    \end{itemize}
    Then the quenched CLT holds. 
    \end{corollary}
    \begin{proof}
       Applying Corollary~\ref{cor} for $s=1$ (which we can due to~\eqref{854eq}), we  see that the assumption~\ref{P assum}  holds with $\mathcal I=\{1\}$ and  $a_n=O(n^{-(1/\gamma-1)})$. Thus, the  conclusion of the theorem follows readily from Theorem \ref{T-CLT}. 

        
    \end{proof}

Next, we discuss the quenched ASIP for random compositions of LSV maps.
    \begin{corollary}\label{asipLSVnew}
Suppose that $\gamma <1/3$ and that~\eqref{eq:mixing_rate} holds with 
\[
 \iota >q \quad \text{and} \quad q> \frac{2(1-\gamma)^2}{\gamma (1-3\gamma)}+2.
\]
Moreover, let $\varphi \colon \Omega \times M \to \mathbb R$ be  a measurable map satisfying:
\begin{itemize}
\item  $\int_M \varphi_\omega \, d\mu_\omega=0$ for $\mathbb P$-a.e. $\omega \in \Omega$, where $\varphi_\omega:=\varphi(\omega, \cdot)$;
        \item for  $\mathbb P$-a.e. $\omega \in \Omega$, $\varphi_\omega$ is Lipschitz and $\esssup_{\omega \in \Omega}\|\varphi_\omega\|_{\Lip}<\infty$.
\end{itemize} Then the quenched ASIP holds.
  \end{corollary}

 \begin{proof}
Choose $p>0$ such that 
\[
q>\frac{p}{2}\left (\frac{1}{\gamma}-1\right)+2 \quad \text{and} \quad p> \frac{4(1-\gamma)}{1-3\gamma}.
\]
Clearly,  $p>4$. It follows from Corollary~\ref{cor} that assumption~\ref{P assum} holds with $\mathcal I=\{2, 2\tilde s\}$, $\tilde s>\frac{1}{1-\frac 2 p}$,  $a_{2, n}=O(n^{-\frac 1 2 (1/\gamma-1)})$ and $a_{2\tilde s, n}=O(n^{-\frac{1}{2\tilde s}(1/\gamma-1)})$.  In particular, $a_{2, n}=O(n^{-b})$, where $b=\frac{1}{2}(1/\gamma-1)>1$ as $\gamma <\frac 1 3$. Moreover, 
\[
\frac{2/p}{b-1}< 1-2/p \quad \text{as $p>\frac{2(1-\gamma)}{1-3\gamma}$.}
\]
For the same reason, by  choosing $\tilde s$ sufficiently close to $\frac{1}{1-\frac 2 p}$, we have that $\sum_{n=1}^\infty a_{2\tilde s, n}<+\infty$.
 The conclusion now follows from Theorem~\ref{T-ASIP} (applied in the case $r=\infty$).

 \end{proof}

\vskip0.2cm

\section*{Appendix A: Proof of Theorem \ref{thm:quenched_ml}}


We closely follow the strategy used in the proof of \cite[Theorem 2.6]{KL2}, thereby showing that proving \eqref{eq:quenched_ml_1} can be reduced to estimating the tail probabilities
$$
\mu_\omega(  \tau_\omega \ge n ) \quad \text{and} \quad 
\tilde{m}( \tau_{   \omega  } \ge n  )
$$
where $\tilde{m}$ denotes the Lebesgue measure on $Y:=[1/2, 1]$ normalized 
to probability and 
$$
\tau_\omega(x) = \inf \{ n \ge 1 \: : \:  T_\omega^n(x) \in Y  \}.
$$

We denote by $\cP_\omega$ the canonical partition (mod $m$) of $M = [0,1]$ into open subintervals
such that $\tau_\omega$ is constant on each $a \in \cP_\omega$. That is, $\cP_\omega$ consists 
of intervals 
$$
(x_{n+1}(\omega),  x_n(\omega) ) \quad \text{and} \quad 
(y_{n+1}(\omega),  y_n(\omega) ), \quad n \ge 1,
$$
where 
\begin{align*}
	x_n(\omega) =   T_\omega^{-n}(1) , \quad y_n(\omega) = 
	\frac{ x_{n-1}(\sigma \omega) + 1 }{2}, \quad n \ge 1.
\end{align*}
Here, the preimages are taken with respect to the left branch of $T_\omega$ and
$x_0(\omega) = 1$. We denote by $\tau_\omega(a)$ the constant value of 
$\tau_\omega$ on $a \in \cP_\omega$.


We start by defining the notion of a regular measure, which serves as a random counterpart of a similar notion in the deterministic setting considered in \cite{KL2}.

For a nonnegative function 
$\psi : Y \to \bR_+$, 
we denote
by $|\psi|_{\LL}$ the Lipschitz seminorm of the logarithm of $\psi$:
\[
|\psi|_{\LL}
= \sup_{y \neq y' \in Y} \frac{|\log \psi(y) - \log \psi(y')|}{d(y,y')}
,
\]
with the conventions $\log 0 = -\infty$ and $\log 0 - \log 0 = 0$. 
Given a measure $\mu$ supported on $Y$
with density $\rho = d \mu / d \tilde m$,
we will often write $|\mu|_{\LL}$ 
for $|\rho|_{\LL}$. Note that 
$$
|\psi|_{\LL} \le ( \inf_Y \psi )^{-1} |\psi|_{\Lip},
$$
where 
$$
|\psi|_{\Lip} = \sup_{y \neq y' \in Y} 
\frac{|\psi(y) - \psi(y')|}{|y - y'|}.
$$

Next, we introduce the family of random induced maps by
$$
F_{\omega, a}(x) = T_\omega^{\tau_\omega(a)}(x), \quad a \in \cP_\omega.
$$
Then (see \cite[Section 3.4]{KL}) there exist $\Lambda > 1$ 
and $K > 0$ such that
\begin{align}\label{eq:KLam}
F_{\omega, a}'(x) \ge \Lambda \quad 
\text{and} \quad 
\biggl| \frac{d(F_{\omega, a})_*(\tilde{m}|_a)}{d \tilde m} \biggr|_{\LL} \le K
\end{align}
hold for all $a \in \cP_\omega$, all 
$x \in a$, and for $\bP$-a.e. $\omega \in \Omega$.

\begin{proposition}\label{prop:k1k2} Fix $K$ and $\Lambda$ such that \eqref{eq:KLam} holds.
	Let
	$$
	K_1 = K + \Lambda^{-1} K_2 \quad \text{and} \quad 
	K_2 > (1 - \Lambda^{-1})^{-1} K.
	$$
	Then $K < K_1 < K_2$ and 
	for each nonnegative measure $\mu$ on $Y$
	with $|\mu|_{\LL} \leq K_2$,
	\[
	\bigl| (  F_{\omega, a} )_* (\mu|_a) \bigr|_{\LL} \leq K_1
	,
	\]
	whenever $a \in \cP_\omega$, $a \subset Y$.
\end{proposition}

\begin{proof}
	The proof is similar to \cite[Proposition 3.1]{KKM}; we provide the details for completeness.
	It suffices to show that for
    $\bP$-a.e. $\omega \in \Omega$, 
	\begin{equation}\label{930c}
	|(F_{\omega, a})_*(\mu |_a)|_{\LL} 
	\le K + \Lambda^{-1} |\mu|_{\LL}.
	\end{equation}
	
	Note that 
	$$
	v_{\omega, a} :=
	\frac{d (F_{\omega, a})_*(\mu |_a) }{ d \tilde m} = (\rho \circ F_{\omega, a}^{-1}) 
	\cdot (F_{\omega, a}^{-1})' 
	= \frac{\rho \circ F_{\omega, a}^{-1}} {F_{\omega,a}' \circ  F_{\omega, a}^{-1} },
	$$
	where $\rho$ is the density 
	of $\mu$ with respect to $\tilde m$.
	Therefore,
	\begin{align*}
		&|\log v_{\omega, a}(x) - \log v_{\omega, a}(x') | \\
		&\le | \log  \rho(F_{\omega, a}^{-1}(x)) 
		-  \log  \rho(F_{\omega, a}^{-1}(x'))  |
		+ | \log F_{\omega,a}'( F_{\omega, a}^{-1}(x)) - \log F_{\omega,a}'( F_{\omega, a}^{-1}(x') )| \\
		&\le (|\mu|_{\LL} \Lambda^{-1} 
		+ K)|x-x'|,
	\end{align*}
    proving~\eqref{930c}.
\end{proof}

\begin{definition} \label{def:reg}
    Fix
$K_1, K_2$ as in Proposition \ref{prop:k1k2}.
	Let $\nu$ be a nonnegative measure on $M$. 
	For $\omega \in \Omega$,
	we say that $\nu$ is regular with respect to $\omega$ 
	if for every $\ell \geq 1$,
	\begin{align}\label{eq:regular}
		\bigl| (T_{ \omega}^\ell )_* (\nu |_{  \{ \tau_{ \omega } = \ell \} }) \bigr|_{\LL}
		\leq K_1
		.
	\end{align}
	Given a function $r : \{0,1,\ldots\} \to [0,\infty)$,
	we say that $\nu$ has tail bound $r$
	with respect to $\omega$,
	if for all $n \geq 0$,
	\begin{align}\label{eq:tail_bound}
		\nu \bigl( \{ x \in M : \tau_\omega(x) \geq n \}\bigr)
		\leq r( n )
		.
	\end{align}
	We say that $\nu$ is regular with tail 
	bound $r$
	w.r.t. $\omega$
	 if both~\eqref{eq:regular}
	and~\eqref{eq:tail_bound} are satisfied. 
\end{definition}

By \cite[Proposition 3.14]{KL}, any probability measure whose density belongs to
$\cC_*$ is regular
w.r.t. $\omega$
for $\bP$-a.e.
$\omega \in \Omega$,
provided we choose 
$K$ in \eqref{eq:KLam} to be sufficiently large 
(depending only on the RDS and the parameter $a$ of the cone $\cC_*(a)$). From now on, we assume that such $K$ has been fixed. 
Then $\mu_\omega$ is regular w.r.t. $\omega$
for $\bP$-a.e.
$\omega \in \Omega$.

Set
\begin{align}\label{eq:h_omega}
u_\omega (n) = \tilde{m}( \tau_{ \omega } \ge n ).
\end{align}
The proof of the following result is essentially the same as that of \cite[Proposition 2.5]{KL2}, and is therefore omitted.

\begin{proposition}\label{prop:regular} Let $k \ge 1$.
	\begin{itemize}
		\item[(a)] For $\bP$-a.e. $\omega \in \Omega$, 
		the measure $\tilde{m}$ is regular w.r.t. 
		$\omega$
		and every measure $\mu$ on $Y$ with $|\mu|_{\LL} \leq K_2$ is regular with tail bound $\mu(Y) e^{K_2} u_\omega$ w.r.t. $\omega$. \smallskip
		\item[(b)] If $\{\mu_j\}$ is a finite or countable collection of measures regular w.r.t.\ $\omega$,
		then $\mu = \sum_{j} \mu_j$ is regular w.r.t.\ $\omega$.  \smallskip
		\item[(c)] 	If $\mu$ is a regular measure w.r.t.\ $\omega$,  then both
		$(T_{\omega}^k)_*\mu$ and $((T_{\omega}^k)_*\mu)|_{M \setminus Y}$ are regular w.r.t.\ $\sigma^k\omega$
		for $\bP$-a.e. $\omega \in \Omega$.
		Moreover, if $n \ge 1$, then for $\bP$-a.e. $\omega \in \Omega$,
		\[
		\Bigl| \bigl( (T^{n}_\omega)_* \mu \bigr) \big|_{Y} \Bigr|_{\LL}
		\leq K_1.
		\]
	\end{itemize}
\end{proposition}

The proof of Theorem \ref{thm:quenched_ml} is carried out in three steps. The first step is to obtain tail bounds for $\mu_\omega$ and $\tilde{m}$.

\subsection*{Step 1: tail bounds for $\mu_\omega$ and $\tilde{m}$}

\begin{proposition} 
There exists a constant $C > 0$ depending only on the RDS such that 
for any $\ve \in (0, 1/2)$ and 
	$\bP$-a.e. $\omega \in \Omega$,
	\begin{align}\label{eq:TB}
		\mu_\omega( \tau_\omega \ge n ) \le  C ( N_\ve(\omega) + 1 )^{ 1 \vee ( 1 / \gamma - 1 ) }      n^{ 1 - 1 / \gamma }.
	\end{align}
	Moreover,
	\begin{align}\label{eq:TB_m}
		u_\omega(n) \le  C S^\sigma_{n}(  \omega )^{ - 1 / \gamma }.
	\end{align}
\end{proposition}

\begin{proof} Throughout the proof, $C$ denotes a generic positive constant depending only on the RDS, whose value may change from one occurrence to the next.
To obtain \eqref{eq:TB}, we observe that
	\begin{align*}
		&\mu_\omega( \tau_\omega \ge n ) =  ( T_{ \sigma^{-1} \omega } )_* \mu_{ \sigma^{-1} \omega } ( \tau_\omega \ge n ) \\
		&=  ( T_{ \sigma^{-1} \omega } )_* [  \mu_{ \sigma^{-1} \omega } |_{  \tau_{ \sigma^{-1} \omega } = 1  }    ] ( \tau_\omega \ge n ) 
		+  ( T_{ \sigma^{-1} \omega } )_* [  \mu_{ \sigma^{-1} \omega } |_{  \tau_{ \sigma^{-1} \omega } > 1  }    ] ( \tau_\omega \ge n ) =: I + II.
	\end{align*}
	By regularity of $\mu_\omega$,  for $\bP$-a.e. $\omega \in \Omega$,
	$$
	\frac{  d (  T_\omega )_* [  \mu_{  \omega } |_{  \{ \tau_\omega = 1 \}  }  ]  }{
		d \tilde{m}
	} \le C,
	$$
	which yields
	$$
	I \le C \tilde{m}( \tau_\omega \ge n ).
	$$
	Moreover,
	\begin{align*}
		II = \mu_{ \sigma^{-1} \omega } (  \{   \tau_{ \sigma^{-1} \omega } > 1   \}  \cap T^{-1}_{ \sigma^{-1} \omega }  \{ \tau_\omega \ge n \}  ) 
		\le \mu_{ \sigma^{-1} \omega }  ( \tau_{ \sigma^{-1} \omega } \ge   n + 1 ).
	\end{align*}
	Thus, we obtain 
	$$
	\mu_\omega( \tau_\omega \ge n )  \le C \tilde{m}( \tau_\omega \ge n ) 
	+ \mu_{ \sigma^{-1} \omega }  ( \tau_{ \sigma^{-1} \omega } \ge   n + 1 ),
	$$
	and, by iteration,
	$$
	\mu_\omega( \tau_\omega \ge n ) \le C \sum_{i=0}^\infty \tilde{m} (  \tau_{  \sigma^{-i} \omega   } \ge n + i  ),
	$$
	for $\bP$-a.e. $\omega \in \Omega$.
	
	If $n \le N_\ve(\omega) + 4/b_0$, we have the trivial estimate
	$$
	\mu_\omega( \tau_\omega \ge n ) \le 1 \le C ( 1 +  N_\ve(\omega) )^{  1 / \gamma - 1 }  n^{- 1/ \gamma + 1}.
	$$
	Then suppose that $n > N_\ve(\omega) + 4/b_0$. Since $\ve < 1/2$, it follows from 
	\eqref{eq:large_n} that $n^{-1} S_n^{ \delta  } (\omega) \ge b_0/2$ for 
	$\delta \in \{ \sigma, \sigma^{-1} \}$.
	
	Arguing as in the proof of \cite[Proposition 3.1]{KL2}, we see that 
	\begin{align}\label{eq:tb_aux}
		\tilde{m} (  \tau_{  \sigma^{-i} \omega   } \ge n + i  ) \le C S^\sigma_{n+i}( \sigma^{- i } \omega )^{ - 1 / \gamma }.
	\end{align}
    Therefore,
	\begin{align*}
		\mu_\omega( \tau_\omega \ge n ) &\le C \sum_{i=0}^\infty S^\sigma_{n+i}( \sigma^{- i } \omega )^{ - 1 / \gamma } \le 
        C \sum_{i=0}^\infty (  S_i^{-\sigma}(\omega) + S_n^{\sigma}(\omega)  
		-1 )^{- 1 / \gamma } \\
        &\le C \sum_{i=0}^\infty (  S_i^{-\sigma}(\omega) +  b_0n / 4   )^{- 1 / \gamma } \\
		&\le C N_\ve( \omega ) n^{- 1 / \gamma } +  C \sum_{i= N_\ve(\omega) + 1 }^\infty (  S_i^{-\sigma}(\omega) +  b_0n / 4   )^{- 1 / \gamma } \\
		&\le C N_\ve( \omega ) n^{- 1 / \gamma } +  C \sum_{i=0 }^\infty (  n + i  )^{- 1 / \gamma } 
		\le C ( N_\ve( \omega ) + 1  ) n^{ 1 - 1 / \gamma },
	\end{align*}
	which completes the proof of \eqref{eq:TB}. Estimate \eqref{eq:TB_m} is obtained from \eqref{eq:tb_aux} by setting $i = 0$.
\end{proof}

\subsection*{Step 2: decomposition of regular measures}

We will deduce Theorem \ref{thm:quenched_ml} from the following result, which is a random
counterpart to \cite[Theorem 2.6]{KL2}.

\begin{theorem}\label{thm:decdec} There exists 
	$\Omega' \subset \Omega$ with $\bP(\Omega') = 1$ 
	such that the following holds for any $\omega \in \Omega'$. 
	Let $\nu$ be a regular 
	probability
	measure on $M$ with tail bound $r$ w.r.t. 
	$\omega$. Then, there exists a decomposition
	$$
	\nu = \sum_{n=1}^\infty \alpha_{\omega, n} \nu_{\omega, n}
	$$
	where $\nu_{\omega, n}$ are probability measures such that $( T_\omega^n )_* \nu_{\omega, n} = \tilde{m}$
	for each $n \ge 1$, and $\alpha_{\omega, n} \in [0,1]$ are numbers such that $\sum_{n \ge 1} \alpha_{\omega, n} = 1$. The sequence $(\alpha_{\omega, n})_{n \ge 1}$
	is fully determined by $K_1$, $K_2$, the RDS, and the tail bound $r$.
	In particular, $(\alpha_{\omega, n})_{n \ge 1}$ does not depend on $\nu$
	in any other way. Moreover, there exists $\tilde{\ve} \in (0, 1/2)$ depending only on the RDS 
	and $K_1, K_2$ such that the following holds for any $\ve \in (0, \tilde{\ve}]$: If there exist $\kappa \ge 1$ and 
	$0 < \eta \le 1/\gamma$ such that for
	every $n \ge 1$, 
	\begin{align}\label{eq:r_estim}
		r(n) \le C_r ( N_\ve(\omega) + 1 )^{ \kappa	 }  n^{ - \eta },
	\end{align}
	then, for
	every $n \ge 1$, 
	$$
	\sum_{j \ge n} \alpha_{\omega, j} \le C C_r ( N_\ve(\omega) + 1)^{  \kappa \vee \eta } n^{ - \eta }.
	$$
	The constant $C$ depends only on the RDS, $K_1, K_2$, 
	and $\eta$.
\end{theorem}

\begin{remark}
	Suppose that $\{ \nu_\omega \}$ and 
	$\{ \nu_\omega' \}$ are two families of probability measures 
	on $M$
	such that, for $\bP$-a.e. $\omega \in \Omega$, 
	$\nu_\omega$ and 
	$\nu'_\omega$ are
	regular 
	with the same tail bound $r_\omega$
	satisfying \eqref{eq:r_estim}.
	Then, as an immediate consequence of 
	Theorem \ref{thm:decdec}, we obtain 
	\begin{align}\label{eq:tv_dec}
		|(T_\omega^n)_*\nu_\omega -(T_\omega^n)_*\nu_\omega'|
		\le 2\sum_{k \ge n} \alpha_{\omega, k}
		\le C C_r ( N_\ve(\omega) + 1)^{  \kappa \vee \eta } n^{ - \eta },
	\end{align}
	for every $n \ge 1$ and for $\bP$-a.e. $\omega \in \Omega$.
	Here, $|\cdot|$ denotes the total variation 
	norm of signed measures.
\end{remark}

\subsubsection*{Proof of Theorem \ref{thm:decdec}}

Let $\Omega_0 \subset \Omega$ be such that 
$\bP(\Omega_0) = 1$ and $\beta( \sigma^k \omega ) \le \beta$
for all $k \in \mathbb{Z}$ whenever $\omega \in \Omega_0$.  

We start with some definitions. Given $\omega \in \Omega_0$, we define 
	\begin{align}\label{eq:h_n}
		u_{\omega, n}(\ell) = C_u ( u_\omega(\ell + n) +  u_{ \sigma \omega }(\ell + n - 1) 
		+ \ldots + u_{ \sigma^n \omega }(\ell)  ),
	\end{align}
	where $C_u = 2 e^{K_2}$. Set 
	$
	\hat{r} (n) = \min \{ 1, r(1), \ldots, r(n)\}$.
	
	Let $X_1, X_2, \ldots$ be random variables 
	on a probability space 
	$(E, \mathcal{E}, P)$
	with values in $\{0, 1,\ldots\}$, such that
	for all $\ell \geq 0$,
	\begin{equation}
		\label{eq:PXj}
		\begin{aligned}
			P(X_1 \geq \ell)
			&= \hat{r}(\ell),
			\\
			P(X_{j+1} \geq \ell \mid X_1, \ldots, X_j)
			& = u_{ \sigma^p \omega ,X_j}(\ell)
			\quad \text{for $j \ge 1$ with  $p = X_1 + \ldots + X_{j-1}$.}
		\end{aligned}
	\end{equation}
	Let $\xi$ be a geometrically distributed random variable 
	on $(E, \mathcal E, P)$ 
	with values in 
	$\{1, 2, \ldots\}$
	and parameter $\theta \in (0,1)$,
	independent of $\{X_j\}$. 
	Define
	\[
	S
	= X_1 + \ldots + X_{\xi}
	.
	\]

	By repeating the argument in the proof of \cite[Theorem 2.6]{KL2} up to 
	\cite[Lemma 4.5]{KL2}, we obtain the following result:
	
	\begin{lemma}\label{lem:decomp} For any $\omega \in \Omega_0$ and 
		for any sufficiently small $\theta$ (depending only on the RDS and 
		$K_1, K_2$), 	
		there exists a decomposition
		$$
		\nu = \sum_{n=1}^\infty P(S = n) \nu_{\omega, n},
		$$
		where $\nu_{\omega, n}$ are probability measures such that 
		$( T_\omega^n )_* \nu_{\omega, n} = \tilde{m}$.
	\end{lemma}
	
	It remains to estimate the tail probabilities $P( S \ge n )$ of $S$ in Lemma 
	\ref{lem:decomp}:
	
	\begin{lemma}\label{lem:tails}	
		There exists $\tilde{\ve} \in (0,1/2)$ depending only on the RDS and $K_1, K_2, \theta$, 
		such that the following holds for any $\ve \in (0, \tilde{\ve}]$.
		For $\bP$-a.e. $\omega \in \Omega$ and every $n \ge 1$,
		\begin{align}\label{eq:return_tails}
			P( S \ge n ) \le C C_r ( N_\ve(\omega) + 1)^{  \kappa \vee \eta } n^{ - \eta }.
		\end{align}
		The constant $C$ depends only on the RDS and $K_1, K_2, \theta$, $\eta$.
		
	\end{lemma}

	\begin{proof}[Proof of Lemma \ref{lem:tails}] We follow closely the argument 
		in the proof of \cite[Proposition 4.6]{KL2}. 
		
		For $j \ge 1$ and $n \ge 1$, 
		we denote $S_j = X_1 + \cdots + X_j$ and decompose 
		\begin{equation}\label{eq:H}
			P(S_{j+1} \geq n) =  \sum_{ \ell = 1}^{n+1} H_\ell,
		\end{equation}
		where
		\begin{align*}
			H_\ell &= P(X_{j+1} \geq n - \ell \mid S_j = \ell) P(S_j = \ell), \quad 1 \le \ell \le n, \\
			H_{n+1} &=  P(S_j > n).
		\end{align*}
		
		From the definition of $u_{\omega, n}$ in \eqref{eq:h_n} it is clear that 
		\[
		u_{  \sigma^{ S_{j-1}  } \omega,   X_j} \leq u_{\omega, S_j}.
		\]
		Using this inequality together with~\eqref{eq:PXj},
		\begin{align}\label{eq:tail_probas}
			P(X_{j + 1} \geq n - \ell \mid S_j = \ell)
			\leq u_{\omega, \ell} (n - \ell)
			= C_u \sum_{i = 0}^{\ell } u_{ \sigma^i \omega }( n - i),
		\end{align}
		for $1 \le \ell \le n$.
		
		Let
		 $\tilde{N}_\ve(\omega) = A + N_\ve(\omega)$ where $A > 0$ is a large integer whose value 
		is specified later. Define $b \in (1/2, 1)$ by
		$$
		b = \frac{ 1 + (1 - \theta)^{ 1 / 2 \eta } }{2}.
		$$
		Further, define
		$$
		\tilde{\ve} = \frac{b_0 ( 1  - b)}{8},
		$$
		and
		$$
		R_j = \sup_{ n \ge 1 }  \left (n^{\eta} P(S_j \geq n) \right ).
		$$
		
		We assume that $\ve \in (0, \tilde\ve]$, and that
		\begin{align}\label{eq:ass_n}
			n \ge \max \biggl\{   \frac{ \tilde{N}_\ve(\omega) }{b}, \frac{4 ( \tilde{N}_\ve(\omega) + 1 ) }{b_0}   \biggr\}.
		\end{align}
		For such $n$, we decompose 
		$$
		P(S_{j+1} \ge n) = \sum_{1 \le \ell \le  \tilde{N}_\ve(\omega) } H_\ell 
		+ \sum_{\tilde{N}_\ve(\omega) < \ell \le   \lfloor  b n \rfloor   } H_\ell + \sum_{ \lfloor  b n \rfloor < \ell \le n + 1} H_\ell =: E_1 + E_2 + E_3,
		$$
		and estimate each term separately. In the rest of the proof we denote by $C$ various constants that depend only on the 
		RDS, $K_1, K_2$, $A$, $\theta$, and $\eta$.
		
		By \eqref{eq:tail_probas},
		\begin{align*} 
			E_1 \le \sup_{ \ell \le \tilde{N}_\ve(\omega) } P(X_{j + 1} \geq n - \ell \mid S_j = \ell)
			\le C_u \sum_{i = 0}^{ \tilde{N}_\ve(\omega) } u_{ \sigma^i \omega }( n - i).
		\end{align*}
		Thus, denoting $\eta_0 = 1/\gamma$, it follows from \eqref{eq:TB_m} that,
		for $\bP$-a.e. $\omega \in \Omega$,
		\begin{align*}
			E_1 &\le C_u C \sum_{i = 0}^{ \tilde{N}_\ve(\omega) }  ( S_{n-i}^\sigma( \sigma^i \omega ) )^{- \eta_0} 
			= C_u C \sum_{i = 0}^{ \tilde{N}_\ve(\omega) }  (  S_n^\sigma(\omega) - S_i^\sigma(\omega)  )^{- \eta_0} \\
			&\le C_u C  \tilde{N}_\ve(\omega) (  S_n^\sigma(\omega) - \tilde{N}_\ve(\omega)  )^{-\eta_0}.
		\end{align*}
		Recall from \eqref{eq:large_n} that $S_n^\sigma(\omega) \ge n b_0 / 2$ for $n \ge N_\ve(\omega)$.
		For $n$ satisfying \eqref{eq:ass_n}, this yields
		\begin{align}\label{eq:estim_e1}
			E_1 \le C_u C  \tilde{N}_\ve(\omega) (  n b_0 / 2 - \tilde{N}_\ve(\omega)  )^{-\eta_0} 
			\le C  \tilde{N}_\ve(\omega)   n^{-\eta_0},
		\end{align}
		and, in particular,
		\begin{align}\label{eq:small_index}
			\sum_{i = 0}^{ \tilde{N}_\ve(\omega) } u_{ \sigma^i \omega }( n - i) 
			\le C  \tilde{N}_\ve(\omega)   n^{-\eta_0},
		\end{align}
		for $\bP$-a.e. $\omega \in \Omega$.
	
		For $E_3$, we have 
		\begin{align}\label{eq:estim_e3}
			E_3 \le P(S_j \ge bn ) \le R_j b^{-\eta} n^{-\eta}.
		\end{align}
		
		For $E_2$, we first estimate 
		\begin{align*}
			\bar{E}_2 &:= \sum_{\tilde{N}_\ve(\omega) < \ell \le   \lfloor  b n \rfloor   } u_{ \sigma^\ell \omega } (n - \ell) 
			P(S_j \ge \ell ) 
			\le C R_j  \sum_{\tilde{N}_\ve(\omega) < \ell \le   \lfloor  b n \rfloor   }  
			(  S_n^\sigma(\omega) - S_\ell^\sigma(\omega)  )^{- \eta_0}
			\ell^{  - \eta } \\
			&\le C R_j  \sum_{\tilde{N}_\ve(\omega) < \ell \le   \lfloor n/2 \rfloor  }  
			(  S_n^\sigma(\omega) - S_\ell^\sigma(\omega)  )^{- \eta_0}
			\ell^{  - \eta } \\
			&+ C R_j  \sum_{\lfloor n/2 \rfloor < \ell \le   \lfloor  b n \rfloor  }  
			(  S_n^\sigma(\omega) - S_\ell^\sigma(\omega)  )^{- \eta_0}
			\ell^{  - \eta } =: \bar{E}_{2,1} +  \bar{E}_{2,2},
		\end{align*}
		for $\bP$-a.e. $\omega \in \Omega$, 
		where \eqref{eq:TB_m} was used in the first inequality.
		
		By \eqref{eq:large_n},
		$$
		S_n^\sigma(\omega) - S_\ell^\sigma(\omega)  \ge (b_0 - \ve)n -  (b_0 + \ve)\ell
		$$
		holds for $\ell \ge \tilde{N}_\ve(\omega)$. Therefore, for $n$ satisfying \eqref{eq:ass_n} we have
		\begin{align*}
			\bar{E}_{2,1} &\le 
			C R_j 
			\sum_{\tilde{N}_\ve(\omega) < \ell \le   \lfloor n/2 \rfloor  }  
			( b_0(n - \ell) - 2 \ve n  )^{- \eta_0} 
			\ell^{-\eta}
			\le L_1(A) R_j   n^{-\eta},
		\end{align*}
		for $\bP$-a.e. $\omega \in \Omega$, 
		where $\lim_{A \to \infty} L_1(A) = 0$, and 
		$\ve \le \tilde{\ve}$ was used to obtain the last inequality.
		On the other hand, 
		\begin{align*}
			\bar{E}_{2,2} &\le C R_j  n^{  - \eta }  \sum_{\lfloor n/2 \rfloor < \ell \le   \lfloor  b n \rfloor  }  
			(   b_0(n - \ell) - 2 \ve n    )^{- \eta_0} \\
			&\le C R_j  n^{ - \eta + 1 - \eta_0  }  ( b_0( 1 - b ) - 2 \ve )^{-\eta_0} \le L_2(A) R_j  n^{  - \eta },
		\end{align*}
		for $\bP$-a.e. $\omega \in \Omega$,
		where $\lim_{A \to \infty} L_2(A) = 0$.
		We conclude that, for $\bP$-a.e. $\omega \in \Omega$,
		\begin{align*}
			\bar{E}_2 \le L(A) R_j  n^{  - \eta },
		\end{align*}
		with $\lim_{A \to \infty} L(A) = 0$.
        Moreover, the function $L$
        depends only on the RDS and $K_1, K_2, \theta$, $\eta$.
		Using \eqref{eq:tail_probas} and summation by parts, it follows that
		\begin{align}\label{eq:estim_e2}
		\begin{split}
			E_2 &\le C_u \sum_{\tilde{N}_\ve(\omega) < \ell \le   \lfloor  b n \rfloor   }  
			P(S_j \ge \ell )  \sum_{i = 0}^{\ell } u_{ \sigma^i \omega }( n - i) \\
			&\le C\bar{E}_2 + C \sum_{0 \le i \le \tilde{N}_\ve(\omega) } u_{ \sigma^i \omega } (n - i) P(S_j \ge \tilde{N}_\ve(\omega)   )\\
			&\le C L(A) R_j  n^{  - \eta }
			+ C  \tilde{N}_\ve(\omega)   n^{-\eta_0},
		\end{split}
		\end{align}
		for $\bP$-a.e. $\omega \in \Omega$,
		where \eqref{eq:small_index} was used in the last inequality. 
		
		Gathering \eqref{eq:estim_e1}, \eqref{eq:estim_e3}, and \eqref{eq:estim_e2},  we conclude that, 
		for $\bP$-a.e. $\omega \in \Omega$, 
		\begin{align*}
			P(S_{j+1} \ge n) &\le  C  \tilde{N}_\ve(\omega)   n^{-\eta_0} + R_j ( \tilde L(A) + b^{-\eta} ) n^{- \eta},
		\end{align*}
		for any $n$ satisfying \eqref{eq:large_n}, where 
        $\tilde L$ is a function 
        depending only on the RDS and $K_1, K_2, \theta$, $\eta$, such that 
        $\lim_{L \to \infty }\tilde{L}(A) = 0$.
        We choose $A$ sufficiently 
		large such that 
		$$
		\tilde L(A) + b^{-\eta} \le (1-\theta)^{-1/2}.
		$$
		Note that this is possible by definition of $b$. It follows that,
		for $\bP$-a.e. $\omega \in \Omega$, 
		$$
		P(S_{j+1} \ge n) \le  C  \tilde{N}_\ve(\omega)   n^{-\eta_0} + R_j (1 - \theta)^{-1/2} n^{-\eta}
		$$
		for $n$ satisfying \eqref{eq:ass_n}. Consequently, for all $n \ge 1$,
		\begin{align*}
		P(S_{j+1} \ge n) &\le  C \tilde{N}_\ve(\omega)   n^{-\eta_0} + R_j (1 - \theta)^{-1/2} n^{-\eta}
		+ C N_\ve(\omega)^\eta n^{-\eta} \\
		&\le C(  N_\ve(\omega) + 1  )^{ 1 \vee \eta  }
		n^{-\eta} + R_j (1 - \theta)^{-1/2} n^{-\eta},
		\end{align*}
		i.e.,
		$$
		R_{j+1} \le C (N_\ve(\omega) + 1)^{ 1 \vee \eta }  + R_j (1 - \theta)^{-1/2}.
		$$
		Recall (see~\eqref{eq:r_estim}) that 
		$$
		P(S_1 \ge n) \le r(n) \le C_r  ( N_\ve(\omega) + 1)^{ \kappa }     n^{ - \eta }.
		$$
		By induction, for $j \ge 1$, 
		and	for $\bP$-a.e. $\omega \in \Omega$, 
		$$
		R_j \le C C_r ( N_\ve(\omega) + 1)^{  \kappa \vee \eta }  (1 - \theta)^{-j/2}.
		$$
		This extends to
		$$
		P(S_j \ge n) \le C C_r 
		( N_\ve(\omega) + 1)^{  \kappa \vee \eta }
		 (1 - \theta)^{-j/2}  n^{ - \eta }.
		$$
		
        Since $\xi$  and $S_j$ are independent and $P(\xi = j) = (1 - \theta)^{j - 1} \theta $, we obtain
		\begin{align*}
			P(S \geq n)
			&= \sum_{j \geq 1} P(S_j \geq n) P(\xi = j) \\
			&\le C C_r ( N_\ve(\omega) + 1)^{  \kappa \vee \eta }  n^{ - \eta }
			\sum_{j \ge 1}  (1 - \theta)^{ j/2 - 1 } \theta  
			\le C C_r ( N_\ve(\omega) + 1)^{  \kappa \vee \eta } n^{ - \eta },
		\end{align*}
		for every $n \ge 1$ and for 
		$\bP$-a.e. $\omega \in \Omega$,
		as wanted.
	\end{proof}
	
	Theorem~\ref{thm:decdec} follows by combining
	Lemmas~\ref{lem:decomp} and~\ref{lem:tails}.

\subsection*{Step 3: final step}

For Lipschitz continuous $g_1, g_2 : M \to \bR$, we write
$$
\tilde{g}_\ell  =  g_\ell  + A_\ell, 
\quad \ell = 1,2,
$$
where $A_\ell = 2\Vert  g_\ell  \Vert_\Lip + 1$.
Then $\tilde{g}_\ell \ge \Vert g_\ell  \Vert_\Lip + 1 \ge \tfrac12 A_\ell$. 
For $0 \le s \le i$, we decompose
\begin{align*}
	\cL_{ \sigma^s \omega }^{i-s} (  h_{ \sigma^s \omega }
	g_1  ) g_2
	&= \cL_{ \sigma^s \omega }^{i-s} ( h_{\sigma^s \omega} 
	\tilde{g}_1
	)
	\tilde{g}_2
	- \cL_{ \sigma^s \omega }^{i-s} ( h_{\sigma^s \omega} 
	\tilde{g}_1
	)
	A_2 - A_1 h_{ \sigma^i \omega } \tilde{g}_2
	+ A_1 A_2 h_{ \sigma^i \omega } \\
	&=: \psi^{(1)}_{ \sigma^i \omega } - \psi^{(2)}_{ \sigma^i \omega } - 
	\psi^{(3)}_{ \sigma^i \omega } + \psi^{(4)}_{ \sigma^i \omega }.
\end{align*}
Set 
$$
d \nu_{\omega, k} =  \frac{\psi^{(k)}_\omega }{ m( \psi^{(k)}_\omega ) } \, dm, \quad 1 \le k \le 4.
$$

Note that
\begin{align*}
	|  \tilde{g}_\ell ( F_{\omega, a}^{-1}  )   |_{\LL} 
	\le \frac{ \Vert  \tilde{g}_\ell  \Vert_{\Lip} }{
	\inf_{[0,1]} \tilde{g}_\ell 
	} \le 3, \quad \forall a \in \cP_\omega, \: \forall \omega \in \Omega.
\end{align*}
Since $\mu_\omega$ is regular w.r.t. $\omega$ for $\bP$-a.e. $\omega \in \Omega$, 
by applying Proposition~\ref{prop:regular}-(c),
we deduce that 
$\nu_{\omega, k}$ is regular 
w.r.t. 
$\omega$ for $\bP$-a.e. $\omega \in \Omega$ and $1 \le k \le 4$,
provided that the constant $K$ in \eqref{eq:KLam} is chosen to be sufficiently large (depending only on the RDS and the parameter $a$ of the cone $\cC_*(a)$).

Next, since
$$
\psi_{\omega}^{(1)}(x)  \le C A_1 A_2 \cL_{ \sigma^{s-i} \omega }^{ i - s }( h_{ \sigma^{s-i} \omega } )(x)
= C A_1 A_2 h_\omega(x)
$$
and 
$$
\psi_{\omega}^{(1)}(x) \ge  \tfrac14 A_1 A_2 \inf_{x \in M} h_\omega(x) \ge C' A_1 A_2 > 0,
$$
we have 
\begin{align*}
	\nu_{\omega, 1}(  \tau_\omega \ge n  ) \le C \mu_\omega(  \tau_\omega \ge n  ).
\end{align*}
Here, $C, C' > 0$ are constants depending only on the RDS and the parameter $a$ of the cone $\cC_*(a)$.
Similarly, for $2 \le k \le 4$, 
\begin{align*}
	\nu_{\omega, k}(  \tau_\omega \ge n  ) \le C \mu_\omega(  \tau_\omega \ge n  ).
\end{align*}
Consequently, the measures $\nu_{\omega, k}$ for $1 \le k \le 4$ have the same tail bound
$$
r_\omega(n) \le  C  ( 1 + N_\ve(\omega) )^{ 1 \vee ( 1 / \gamma - 1 ) }   n^{ 1 - 1 / \gamma }
$$
w.r.t. $\omega$ for $\bP$-.a.e $\omega \in \Omega$.

Denote by $\rho_{\omega, k} = \psi^{(k)}_\omega / m( \psi^{(k)}_\omega ) $ the density of $\nu_{\omega, k}$. Then, 
\begin{align*}
	&\Vert \cL_{ \sigma^i \omega }^{ j - i }(  \cL_{ \sigma^s \omega }^{i-s} (  h_{ \sigma^s \omega } g_1 ) g_2 )  
	- m(   \cL_{ \sigma^s \omega }^{i-s} (  h_{ \sigma^s \omega } g_1 ) g_2    ) h_{ \sigma^j \omega }
	\Vert_{ L^1(m) }\\
	&\le \sum_{k=1}^4 \Vert \cL_{\sigma^i \omega }^{j - i}  \psi^{(k)}_{\sigma^i \omega } - m( \psi^{(k)}_{\sigma^i \omega }) h_{ \sigma^j \omega }  \Vert_{L^1(m)}
	\le \sum_{k=1}^4 
	m( \psi^{(k)}_{ \sigma^i \omega  } )
	\Vert \cL_{\sigma^i \omega }^{j - i} (  \rho_{ \sigma^i \omega, k } - h_{ \sigma^i \omega }  ) \Vert_{L^1(m)} \\
	&\le  \sum_{k=1}^4 
	m( \psi^{(k)}_{ \sigma^i \omega  } )
	| ( T_{ \sigma^i \omega }^{j-i} )_*( \nu_{\sigma^i\omega, k}  ) 
	- ( T_{ \sigma^i \omega }^{j-i} )_*( \mu_{\sigma^i\omega}  )   |,
\end{align*}
where
$|\cdot|$ denotes the total variation of signed measures.
Since $|m( \psi^{(k)}_{ \sigma^i \omega  } )| \le CA_1 A_2$,
an application of \eqref{eq:tv_dec} now yields the upper bound
\begin{align*}
    &\bigl\Vert  \bigl[  L_{  \sigma^i \omega }^{j-i} \bigl(g_2  L_{ \sigma^s \omega }^{i-s}   (g_1 \bigr)  \bigr]_{ \sigma^j \omega }  \bigr\Vert_{  L^1(  \mu_{ \sigma^j \omega } ) } \\
	&=\Vert 
	\cL_{ \sigma^i \omega }^{ j - i }(  \cL_{ \sigma^s \omega }^{i-s} (  h_{ \sigma^s \omega } g_1 ) g_2 )  
	- m(   \cL_{ \sigma^s \omega }^{i-s} (  h_{ \sigma^s \omega } g_1 ) g_2    ) h_{ \sigma^j \omega }
	\Vert_{ L^1(m) } \\
	&\le C A_1 A_2
	( 1 + N_\ve( \sigma^i \omega) )^{ 1 \vee ( 1 / \gamma - 1 ) } 
	 (j - i)^{-1/\gamma + 1},
\end{align*}
for $\bP$-a.e. $\omega \in \Omega$, provided that $\ve \in (0, \tilde{\ve}]$. This completes
the proof of Theorem \ref{thm:quenched_ml}.

\section*{Appendix B: Some auxiliary lemmas}
The following result is established in~\cite[Lemma 2]{DH}.
\begin{lemma}\label{CMPlemma}
Let $(\Omega, \mathcal F, \mathbb P)$ be a probability space, $\sigma \colon \Omega \to \Omega$ an invertible $\mathbb P$-preserving transformation and $B\in L^q(\Omega, \mathcal F, \mathbb P)$ for $q>0$. Furthermore, let $(a_n)_{n\in \mathbb N}$ be a sequence of positive numbers such that $\sum_{n=1}^\infty a_n^q<\infty$. Then there is $R\in L^q(\Omega, \mathcal F, \mathbb P)$ such that 
\[
B(\sigma^{-n}\omega)\le R(\omega)a_n^{-1},
\]
for $\mathbb P$-a.e. $\omega \in \Omega$ and $n\in \mathbb N$. In particular, for every $\delta>0$ there is $R_\delta \in L^q(\Omega, \mathcal F, \mathbb P)$ such that 
\[
B(\sigma^{-n}\omega)\le R_\delta(\omega)n^{1/q+\delta},
\]
for $\mathbb P$-a.e. $\omega \in \Omega$ and $n\in \mathbb N$. 

\end{lemma}
Next, we shall need the following result from~\cite[Lemma 9]{DH1}.
\begin{lemma}\label{proclemma}
 Let $Y_i$, $i=1, 2, \ldots$ be a sequence of zero mean random variables on some probability space. Suppose that there are constants $C,\eta>0$ such that for all $n,m$ we have
 $$\left\|\sum_{j=m+1}^{n+m}Y_j\right\|_{L^2}^2\leq Cn(n+m)^{2\eta}.
$$
 Then, almost surely for every $\delta>0$ we have 
$$
\left|\sum_{j=1}^n Y_j\right|=O(n^{1/2+\eta+\delta}).
$$

\end{lemma}

\section*{Acknowledgements}
This paper has been funded by the European Union – NextGenerationEU-StatDinMatAn-uniri-iz-25-108 (Davor Dragi\v cevi\' c). J. Lepp\"anen was supported by JSPS via the project LEADER. He thanks Alexey Korepanov for valuable correspondence.  The authors thank the referees for their valuable comments and suggestions.

\end{document}